\documentclass[12pt]{amsart}

\usepackage{amsmath,amssymb,amsthm,textcomp}
\usepackage{amsfonts,graphicx}
\usepackage[mathscr]{eucal}
\usepackage{color}
\usepackage{multicol}
\usepackage[margin=1in]{geometry}

\theoremstyle{definition}

\allowdisplaybreaks
\numberwithin{equation}{section}

\newcommand{\ci}{\perp\!\!\!\perp}

\newcommand{\ncom}{\newcommand}

\ncom{\beq}{\begin{equation}}
\ncom{\eeq}{\end{equation}}
\ncom{\bea}{\begin{eqnarray*}}
\ncom{\eea}{\end{eqnarray*}}
\ncom{\beqa}{\begin{eqnarray}}
\ncom{\eeqa}{\end{eqnarray}}
\ncom{\nno}{\nonumber}
\ncom{\non}{\nonumber}
\ncom{\ds}{\displaystyle}
\ncom{\half}{\frac{1}{2}}
\ncom{\mbx}{\makebox{.25cm}}
\ncom{\hs}{\mbox{\hspace{.25cm}}}
\ncom{\rar}{\rightarrow}
\ncom{\Rar}{\Rightarrow}
\ncom{\noin}{\noindent}
\ncom{\bc}{\begin{center}}
\ncom{\ec}{\end{center}}
\ncom{\sz}{\scriptsize}
\ncom{\rf}{\ref}
\ncom{\s}{\sqrt{2}}
\ncom{\sgm}{\sigma}
\ncom{\Sgm}{\Sigma}
\ncom{\psgm}{\sigma^{\prime}}
\ncom{\dt}{\delta}
\ncom{\Dt}{\Delta}
\ncom{\lmd}{\lambda}
\ncom{\Lmd}{\Lambda}
\ncom{\Th}{\Theta}
\ncom{\e}{\eta}
\ncom{\eps}{\epsilon}
\ncom{\pcc}{\stackrel{P}{>}}
\ncom{\lp}{\stackrel{L_{p}}{>}}
\ncom{\dist}{{\rm\,dist}}
\ncom{\sspan}{{\rm\,span}}
\ncom{\re}{{\rm Re\,}}
\ncom{\im}{{\rm Im\,}}
\ncom{\sgn}{{\rm sgn\,}}
\ncom{\ba}{\begin{array}}
\ncom{\ea}{\end{array}}
\ncom{\hone}{\mbox{\hspace{1em}}}
\ncom{\htwo}{\mbox{\hspace{2em}}}
\ncom{\hthree}{\mbox{\hspace{3em}}}
\ncom{\hfour}{\mbox{\hspace{4em}}}
\ncom{\vone}{\vskip 2ex}
\ncom{\vtwo}{\vskip 4ex}
\ncom{\vonee}{\vskip 1.5ex}
\ncom{\vthree}{\vskip 6ex}
\ncom{\vfour}{\vspace*{8ex}}
\ncom{\norm}{\|\;\;\|}
\ncom{\integ}[4]{\int_{#1}^{#2}\,{#3}\,d{#4}}
\ncom{\vspan}[1]{{{\rm\,span}\{ #1 \}}}
\ncom{\dm}[1]{ {\displaystyle{#1} } }
\ncom{\ri}[1]{{#1} \index{#1}}

\newtheorem{theorem}{\bf Theorem}[section]
\newtheorem{remark}{\bf Remark}[section]

\newtheorem{lemma}{Lemma}[section]
\newtheorem{corollary}{Corollary}[section]
\newtheorem{example}{Example}[section]
\newtheorem{definition}{Definition}[section]
\newtheoremstyle
    {remarkstyle}
    {}
    {11pt}
    {}
    {}
    {\bfseries}
    {:}
    {     }
    {\thmname{#1} \thmnumber{#2} }

\theoremstyle{remarkstyle}

\def\eps{\varepsilon}

\def\Z{{\mathbb Z}}

\begin{document}
\title{\Large M\lowercase{arginal} L\lowercase{og-linear} P\lowercase{arameters} \lowercase{and their} C\lowercase{ollapsibility for} C\lowercase{ategorical} D\lowercase{ata}}
\author[Sayan Ghosh]{S. Ghosh}
\address{Sayan Ghosh, Department of Statistics,
 University of Haifa, Haifa 3498838, Israel.}
 \email{sayan38@gmail.com}
\author{P. Vellaisamy}
\address{P. Vellaisamy, Department of Mathematics,
Indian Institute of Technology Bombay, Powai, Mumbai 400076, India.}
\email{pv@math.iitb.ac.in}
\begin{abstract}
\noindent We consider marginal log-linear models for parameterizing distributions on multidimensional contingency tables. These models generalize ordinary log-linear and multivariate logistic models, besides several others. First, we obtain some characteristic properties of marginal log-linear parameters. Then we define collapsibility and strict collapsibility of these parameters in a general sense. Several necessary and sufficient conditions for collapsibility and strict collapsibility are derived based on simple functions of only the cell probabilities, which are easily verifiable. These include results for an arbitrary set of marginal log-linear parameters having some common effects. The connections of strict collapsibility to various forms of independence of the variables are explored. We analyze some real-life datasets to illustrate the above results on collapsibility and strict collapsibility. Finally, we obtain a result relating parameters with the same effect but different margins for an arbitrary table, and demonstrate smoothness of marginal log-linear models under collapsibility conditions.  
\end{abstract}

\maketitle
\vspace*{-0.7cm}

\section{Introduction}
Various models for multidimensional contingency tables have been proposed by imposing restrictions on marginal or conditional distributions, especially in the context of longitudinal and causal models. Some references include Liang, Zeger and Qaqish (1992), Becker (1994) and Lang and Agresti (1994). In this paper, we consider the class of marginal log-linear (MLL) models introduced by Bergsma and Rudas (2002), which generalize ordinary log-linear models, multivariate logistic models (McCullagh and Nelder (1989), Glonek and McCullagh (1995)), and the mixture of these models (Glonek (1996)). The MLL parameters are computed from marginals of the joint distribution and are characterized by two subsets of the variables -- the relevant marginal and the effect (a subset of the marginal).

MLL parameterizations provide an elegant and flexible way to parameterize a multivariate discrete probability distribution. Useful submodels can be induced by setting some of the parameters to 0, or more generally by restricting attention to a linear or affine subset of the parameter space. If these zero parameters can be embedded into a larger smooth parameterization of the joint distribution, then the model defined by the conditional independence constraints is a curved exponential family, and therefore possesses good statistical properties. A smooth parameterization implies the applicability of standard asymptotic theory and simplifies interpretation. This approach was applied by Rudas {\it et al.} (2010) and Forcina {\it et al.} (2010) for conditional independence models, and Evans and Richardson (2013) to some classes of graphical models. More recently, Evans (2015) demonstrated smoothness of certain MLL parameterizations.

In this paper, we use marginal models (see Bergsma, Croon and Hagenaars (2009)) for the analysis of a multidimensional contingency table, which may be quite involved for tables of high dimension. So, it is often useful and convenient to reduce the dimension of the table and examine the condensed (summed over levels of certain variables) table, for example, if the original table is sparse or the observed cell counts are small (Ducharne and Lepage (1986)). However, in a condensed table, some extraneous association between the remaining variables may be introduced. Also, any original relationship between certain variables may be lost and/or the monotonicity of dependence among some variables may be reversed. Some references for this paradox, commonly known as the Simpson's paradox, are Simpson (1951), Cox and Wermuth (2003) and Vellaisamy (2012). Hence, it is of practical importance to identify various conditions for collapsibility of a given table, that is, if the table can be condensed without affecting certain interaction parameters for the remaining variables. With the huge volume of data available nowadays, collapsibility may be viewed as a `dimension reduction' technique for data condensation. The study of collapsibility is also important because all multivariate statistical analysis happens on a marginal of a larger table, where one mostly does not know the variables upon which collapsing occurred. 

Wermuth (1987) studied parametric collapsibility with respect to odds ratio and relative risk, and Guo and Geng (1995) discussed collapsibility conditions for logistic regression coefficients. Whittemore (1978) obtained some necessary and sufficient conditions for collapsibility and strict collapsibility for a $n$-dimensional table. However, due to arbitrary functional representations and the algebraic approach, the results and their proofs are quite involved and non-intuitive. Vellaisamy and Vijay (2007) studied collapsibility for a multidimensional contingency table using ordinary log-linear parameters, which have simple closed-form expressions. Some related references are  Vellaisamy and Vijay (2009), and  Vellaisamy and Vijay (2010). Note that for a given multidimensional contingency table, ordinary log-linear parameters can be defined only for the full table, while MLL parameters can be defined within any marginal of the table (see Bergsma and Rudas (2002)). In this paper, we obtain various results related to collapsibility and strict collapsibility of MLL parameters for such tables, which generalize the ordinary log-linear parameters in the sense described above. 

The remaining paper is organized as follows. In Section 2, we describe various terms and notations that are used throughout the paper. We discuss the concept of a marginal log-linear parameterization introduced by Bergsma and Rudas (2002), and define the MLL parameters. We use a simple and intuitive expression for such parameters involving only subsets of an effect defined within a marginal. Some important properties of these parameters are derived. In Section 3, we give a general definition of collapsibility by considering two arbitrary marginals of a contingency table and provide a set of equivalent conditions for collapsibility of MLL parameters. We obtain necessary and sufficient conditions for collapsibility with respect to a set of MLL parameters having some common effects. All the above conditions generalize those obtained by Vellaisamy and Vijay (2007) for ordinary log-linear parameters. Similar results for strict collapsibility of MLL parameters are derived. Interestingly, all results in this section are expressed in terms of simple functions of the cell probabilities whose MLE's can be easily computed either in closed forms under certain models or using iterative procedures. In Section 4, we explore the relationship of strict collapsibility with conditional, joint and mutual independence of variables for a multidimensional table. New necessary and sufficient conditions are obtained in each case. From a theoretical perspective, these conditions characterize various forms of independence among variables in a contingency table in terms of strict collapsibility of MLL parameters. From a practical perspective, we can infer independence relations among the variables by verifying the simple collapsibility conditions or equivalently collapse larger tables into smaller ones if the variables satisfy the independence relations. We provide various real-life data analysis examples in Sections 3 and 4 to illustrate the results therein. In Section 5, we obtain a new result relating parameters having a common effect but defined within different marginals of an arbitrary table. This result is then used to show the existence of a smooth MLL parameterization or curved exponential family under collapsibility conditions and obtain the same. Also, a sufficient condition for collapsibility of MLL parameters in a multidimensional table is provided using the above result. Section 6 mentions some concluding remarks. All proofs are included in the Appendix.

\section{Marginal log-linear parameters}
In this section, we consider marginal log-linear (MLL) parameters for a multidimensional contingency table and discuss related notations and concepts. Also, we establish some useful and interesting properties of these parameters.

\subsection{Notations and Definitions} First, we introduce some terminology that will be used subsequently. Let $V$ be a finite index set. For $v\in V$, let $X_{v}$ be a categorical variable with levels $x_{v}$ in $\mathfrak{X}_{v}=\{0,1,\ldots,|\mathfrak{X}_{v}|-1\}$ (say). Throughout the paper we assume $V$ to be fixed, that is, the full set of variables is known. We denote $\mathfrak{X}_{A}=\times_{v\in A}(\mathfrak{X}_{v})$, $X_{A}=\{ X_{v}\mid v\in A\}$ and $x_{A}=\{x_{v}\mid v\in A\}$ where $\emptyset\neq A\subseteq V$. Also, let $\tilde{\mathfrak{X}}_{v}=\{0,1,\ldots,|\mathfrak{X}_{v}|-2\}$ and $\tilde{\mathfrak{X}}_{A} = \times_{v\in A}(\tilde{\mathfrak{X}}_{v})$. The marginal distribution of $X_{A}$ is denoted by $p_{A}(x_{A})$ while the conditional distribution of $X_{A}|X_{B}$ is $p_{A|B}(x_{A}|x_{B})$ for disjoint $A,B\subseteq V$. It is assumed that $p_{V}(x_{V})>0$. 

The Cartesian product $\mathfrak{X}_{V}$ is called a $|V|$-dimensional contingency table. Let $x = (x_{1},\ldots,\\x_{|V|})$ be a cell of the table with cell frequency $n(x)\geq 0$ and cell probability $p(x)>0$. A marginal cell probability for the marginal table $\mathfrak{X}_{A}$ is given by $p_{A}(x_{A})=\sum_{j\in\mathfrak{X}_{V}:j_{A}=x_{A}}p(j)$. Let $\mathcal{F}=\{p(x):p(x)>0,\sum_{x\in\mathfrak{X}_{V}}p(x)=1\}$ be the strictly positive probability simplex of dimension $k=\prod_{v\in V}|\mathfrak{X}_{v}|-1$ on $\mathfrak{X}_{V}$. A function $\theta:\mathcal{F}\rightarrow\mathbb{R}^{k},k\geq 1$ is called a {\it parameter} of $\mathcal{F}$. For an open set $D\subseteq\mathbb{R}^{k}$, $\theta:\mathcal{F}\rightarrow D$ is a {\it smooth parameterization} of $\mathcal{F}$ if it is a homeomorphism onto $D$, is twice continuously differentiable and its Jacobian has full rank $k$ everywhere. If $\mathcal{F}$ belongs to an exponential family, then a model $\mathcal{G}\subseteq\mathcal{F}$ is called {\it curved exponential} if it has a smooth parameterization. 

Now we discuss some concepts related to a MLL parameterization that are used later.
\begin{definition}\label{def1.1}
Let $\mathcal{P} = \{(P,Q)\}$ be a collection of ordered pairs of subsets of $V$ such that $P\subseteq Q\subseteq V$. Define
\begin{equation}\label{eq1.1}
\mathcal{Q}_{\mathcal{P}}=\{Q|(P,Q)\in\mathcal{P}~\textrm{for some}~P\subseteq Q\}
\end{equation}  
to be the collection of margins in $\mathcal{P}$. If $\mathcal{Q}_{\mathcal{P}}=\{Q_{1},\ldots,Q_{t}\}$, then define $\mathbb{P}_{i}=\{P|(P,Q_{i})\in\mathcal{P}\}$ to be the collection of effects defined within the margin $Q_{i}$. We call $\mathcal{P}$ {\it hierarchical} if there is an ordering on $\mathcal{Q}_{\mathcal{P}}$ such that $i<j\Rightarrow Q_{j}\not\subseteq Q_{i}$ (the sequence of margins is non-decreasing) and $P\in\mathbb{P}_{j}\Rightarrow P\not\in\mathbb{P}_{i}$ (every effect is contained only within the first margin of which it is a subset). We call $\mathcal{P}$ {\it complete} if for every non-empty $P$ there is exactly one $Q_{i}\in\mathcal{Q}_{\mathcal{P}}$ such that $(P,Q_{i})\in\mathcal{P}$ (all effects are considered, each defined only once in a specific margin), which implies the last $Q_{i}$ in the ordering is $V$. 
\end{definition}
An ordered pair $(P,Q_{i})$ represents a MLL interaction over the effect $P$ within the margin $Q_{i}$. Next, we formally define a MLL parameter. For equivalent definitions, see Bergsma and Rudas (2002) or Evans and Richardson (2013).
\begin{definition}\label{def1.3}
For $L\subseteq M\subseteq V$ and $x_{L}\in\mathfrak{X}_{L}$, let
\begin{equation}\label{eq1.3}
\nu_{L}^{M}(x_{L}) = \frac{1}{|\mathfrak{X}_{M\backslash L}|}\sum_{j_{M}\in\mathfrak{X}_{M}:j_{L}=x_{L}}\log p_{M}(j_{M})
\end{equation}
and
\begin{align}\label{eq1.4}
\lambda_{L}^{M}(x_{L}) &= \sum_{L'\subseteq L}(-1)^{|L\backslash L'|}\nu_{L'}^{M}(x_{L'}) \nonumber \\
&= \sum_{L'\subseteq L}(-1)^{|L\backslash L'|}\frac{1}{|\mathfrak{X}_{M\backslash L'}|}\sum_{j_{M}\in\mathfrak{X}_{M}:j_{L'}=x_{L'}}\log p_{M}(j_{M}).
\end{align} 
Then $\lambda_{L}^{M}(x_{L})$ is called a MLL parameter.
\end{definition}
The expression in (\ref{eq1.4}) involves only subsets of an effect defined within a marginal and marginal cell probabilities. Using M\"{o}bius inversion (see Charalambides (2002)) for $L\subseteq M\subseteq V$ and $x_{L}\in\mathfrak{X}_{L}$, we have  from (\ref{eq1.4})
\begin{equation}\label{eq1.13}
	\lambda_{L}^{M}(x_{L}) = \sum_{L'\subseteq L}(-1)^{|L\backslash L'|}\nu_{L'}^{M}(x_{L'})\Leftrightarrow\nu_{L}^{M}(x_{L})=\sum_{L'\subseteq L}\lambda_{L'}^{M}(x_{L'}).
\end{equation} 

Let $\lambda_{\mathcal{P}}=\{\lambda_{L}^{M}(x_{L})|(L,M)\in\mathcal{P},x_{L}\in\mathfrak{X}_{L}\}$ be the MLL parameterization corresponding to $\mathcal{P}$ (see Definition \ref{def1.1}) and $\lambda_{L}^{M}$ denote the collection $\{\lambda_{L}^{M}(x_{L})|x_{L}\in\mathfrak{X}_{L}\}$. Important special cases of $\lambda_{\mathcal{P}}$ are the classical or ordinary log-linear parameters denoted by $\{\lambda_{L}^{V}|L\subseteq V\}$ and the multivariate logistic parameters denoted by $\{\lambda_{L}^{L}|L\subseteq V\}$. To avoid redundancy due to identifiability constraints (see Lemma \ref{lem1.1}), consider only $x_{L}\in\tilde{\mathfrak{X}}_{L}$ in  $\lambda_{\mathcal{P}}$ so that the MLL parameters $\tilde{\lambda}_{L}^{M}(x_{L}):\mathcal{F}\rightarrow\mathbb{R}^{|\tilde{\mathfrak{X}}_{V}|}$ are elements of the collection $\tilde{\lambda}_{\mathcal{P}}$. Theorem 2 of Bergsma and Rudas (2002) showed that if $\mathcal{P}$ is hierarchical and complete, then $\tilde{\lambda}_{\mathcal{P}}$ is a smooth parameterization of $\mathcal{F}$ (the saturated model). 

\subsection{Properties of marginal log-linear parameters}

Here, we obtain various useful properties of MLL parameters, some of which are used subsequently in the paper. 

The result below is mentioned without proof in Evans and Richardson (2013). It provides an identifiability constraint for a MLL parameter and states that the sum over all levels of any argument of a MLL parameter is $0$.
\begin{lemma}\label{lem1.1}
For every $\emptyset\neq L\subseteq M$, the MLL parameters $\lambda_{L}^{M}$ satisfy the constraint  
\begin{equation}\label{eq1.8}
\sum_{x_{v}}\lambda_{L}^{M}(x_{L})=\sum_{x_{v}}\lambda_{L}^{M}(x_{v},x_{L\backslash\{v\}})=0\quad\forall~x_{L}\in\mathfrak{X}_{L},~v\in L.
\end{equation}
\end{lemma}

The following result states a characteristic property of MLL parameters, which appears as Definition 2 in Evans and Richardson (2013). 
\begin{lemma}\label{lem1.11}
For each $\emptyset\neq M\subseteq V$, the MLL parameters $\lambda_{L}^{M}$, subject to the identifiability constraint in Lemma \ref{lem1.1}, satisfy the identity
\begin{equation}\label{eq1.11a}
\log p_{M}(x_{M})=\sum_{L\subseteq M}\lambda_{L}^{M}(x_{L})\quad\forall~x_{L}\in\mathfrak{X}_{L},~x_{M}\in\mathfrak{X}_{M}.
\end{equation}
\end{lemma}
Consider now two sets of MLL parameters, where parameters for one set are defined within a certain margin and those for the other set are defined within a different margin. Then we have a new and interesting result which states that if the sums of the parameters for both sets are equal, then the individual parameters corresponding to the two distinct margins are also equal.  
\begin{lemma}\label{lem1.1a}
For $\emptyset\neq L\subseteq N\subset M$, we have
\begin{equation}\label{eq1.10a}
\sum_{L'\subseteq L}\lambda_{L'}^{M}(x_{L'})=\sum_{L'\subseteq L}\lambda_{L'}^{N}(x_{L'})\Leftrightarrow\lambda_{L'}^{M}(x_{L'})=\lambda_{L'}^{N}(x_{L'})\quad\forall~ L'\subseteq L, x_{L'}\in\mathfrak{X}_{L'}. 
\end{equation}	
\end{lemma}
The next lemma provides a new connection between $\nu_{L}^{M}$ defined in (\ref{eq1.3}) and the MLL parameters $\lambda_{L}^{M}$, thereby stating an equivalent condition for Lemma \ref{lem1.11} to hold. It also generalizes a result of Vellaisamy and Vijay (2007) for ordinary log-linear parameters. 
\begin{lemma}\label{lem1.2}
The MLL parameters $\lambda_{L}^{M}$ satisfy
\begin{equation}\label{eq1.11}
\log p_{M}(x_{M})=\sum_{L\subseteq M}\lambda_{L}^{M}(x_{L})
\end{equation} 
if and only if
\begin{equation}\label{eq1.12}
\nu_{A}^{M}(x_{A})=\sum_{L\subseteq A}\lambda_{L}^{M}(x_{L})\quad\forall~A\subseteq M,
\end{equation}
where $\nu_{L}^{M}$ is given by (\ref{eq1.3}).
\end{lemma}
The following result deals with the restrictions on MLL parameters for conditional independence models (see Lemma 1 of Rudas {\it et al.} (2010) or Lemma 2 of Evans and Richardson (2013)). Henceforth in this paper, we sometimes denote $S\cup T$ as $ST$ for convenience.
\begin{lemma}\label{lem1.3}
Consider three disjoint subsets $A$, $B$ and $C$ of $V$, where $C$ may be empty. Then $X_{A}\ci X_{B}|X_{C}$ if and only if
\begin{equation}\label{eq1.14}
\lambda_{A'B'C'}^{ABC}=0\quad\textrm{for every}~\emptyset\neq A'\subseteq A,~\emptyset\neq B'\subseteq B,~C'\subseteq C. 
\end{equation}
\end{lemma}
 If $C=\emptyset$ in Lemma \ref{lem1.3}, then we have marginal independence between $X_{A}$ and $X_{B}$. The result in this case was proved for multivariate logistic parameters by Kauermann (1997). 

Finally, we provide a new, explicit expression for certain conditional distribution parameters defined as sums of MLL parameters having a common effect. It also shows the relationship between the two classes of parameters.
\begin{lemma}\label{lem1.5}
For $\emptyset\neq L\subseteq M\subseteq V$ with $N=M\backslash L$, define $\kappa_{L|N}(x_{L}|x_{N})=\sum_{L\subseteq A\subseteq M}\lambda_{A}^{M}(x_{A})$. Then
\begin{equation}\label{eq1.21}
 \kappa_{L|N}(x_{L}|x_{N})=\sum_{L'\subseteq L}(-1)^{|L\backslash L'|}\frac{1}{|\mathfrak{X}_{L\backslash L'}|}\sum_{\substack{j_{M}\in\mathfrak{X}_{M}\\j_{L'\cup N}=x_{L'\cup N}}}\log p_{M}(j_{M}).
\end{equation}
\end{lemma}
By Lemma 5 of Evans and Richardson (2013), for fixed $M$ and $L\subseteq M$, the collection of MLL parameters $\{\tilde{\lambda}_{A}^{M}(x_{A})|L\subseteq A\subseteq M,x_{A}\in\tilde{\mathfrak{X}}_{A}\}$ in Lemma \ref{lem1.5} together with the $(|L|-1)$-dimensional marginal distributions of $X_{L}$ conditional on $X_{M\backslash L}$ smoothly parameterize the conditional distribution of $X_{L}$ given $X_{M\backslash L}$.

\section{Collapsibility and Strict Collapsibility}
In this section, we establish some necessary and sufficient conditions for collapsibility and strict collapsibility of a multidimensional contingency table with respect to MLL parameters. These conditions involve only simple functions of the cell probabilities, which can be easily computed from the table. Closed-form MLE's of the cell probabilities exist under certain types of models (for example, conditional independence models), otherwise iterative procedures like the iterative proportional fitting (IPF) algorithm may be used to obtain the MLE's. 

As mentioned in Section 2, note that a classical or ordinary log-linear parameter denoted by $\lambda_{L}^{V}$ ($L\subseteq V$) is a special case of a MLL parameter denoted by $\lambda_{L}^{M}$ ($L\subseteq M\subseteq V$). For example, if $V=\{1,2,3\}$, then the ordinary log-linear parameterization (see Bergsma and Rudas (2002)) is unique and is given by
\begin{equation}\label{eq1}
\{\lambda_{L}^{123} | L\subseteq \{1,2,3\}\} = \{\lambda_{\emptyset}^{123}, \lambda_{1}^{123}, \lambda_{2}^{123}, \lambda_{3}^{123}, \lambda_{12}^{123}, \lambda_{13}^{123}, \lambda_{23}^{123}, \lambda_{123}^{123} \}.
\end{equation} 
However, a MLL parameterization is not unique. It depends on the choice of the sequence of margins (subsets of $V$) and the order in which the effects are defined within the margins. In the above case, if $\mathcal{M}=\{12,13,123\}$ is the sequence of margins, then a corresponding MLL parameterization would be
\begin{equation}\label{eq2}
\{\lambda_{L}^{M} | L\subseteq M\in\mathcal{M}\} = \{\lambda_{\emptyset}^{12}, \lambda_{1}^{12}, \lambda_{2}^{12}, \lambda_{12}^{12}, \lambda_{3}^{13}, \lambda_{13}^{13}, \lambda_{23}^{123}, \lambda_{123}^{123} \}.
\end{equation} 
Indeed, (\ref{eq2}) represents the complete and hierarchical MLL parameterization (see Definition \ref{def1.1}) for the above sequence of margins. From (\ref{eq1.4}), note that MLL parameters are functions of the marginal probabilities $p_{M}$, while ordinary log-linear parameters are functions of the joint probability $p_{V}$ in a table. Also, observe that for a given multidimensional contingency table ($V$ is fixed), ordinary log-linear parameters can be defined only for the marginal $V$ corresponding to the full table, while MLL parameters can be defined for any arbitrary marginal $M\subseteq V$ of the full table. From Definitions \ref{def2.1} and \ref{def3.1} below, it is thus meaningful to discuss collapsibility and strict collapsibility of MLL parameters rather than ordinary log-linear parameters for such a table. Whittemore (1978) and Vellaisamy and Vijay (2007), among others, studied collapsibility and strict collapsibility of ordinary log-linear parameters for a multidimensional contingency table. However, this is not technically correct as discussed above. Therefore our results in this section appropriately modify the results of Vellaisamy and Vijay (2007) (see Theorems 3.1-3.3 for collapsibility, and Lemma 4.1 alongwith Theorems 4.1-4.2 for strict collapsibility in their paper). Moreover, we generalize the above results  by considering two arbitrary marginals of a multidimensional contingency table instead of one of them being the full table usually considered in the literature.

\subsection{Collapsibility of MLL parameters} Bergsma and Rudas (2002) defined a complete table $\mathfrak{X}_{V}$ to be {\it collapsible} into the marginal table $\mathfrak{X}_{M}$ with respect to $\emptyset\neq L\subseteq M$ if $\lambda_{L}^{V}(x_{L})=\lambda_{L}^{M}(x_{L})$ for all $x_{L}\in\mathfrak{X}_{L}$. This condition implies that the amount of information about the interaction among the variables in $L$ is the same in $\mathfrak{X}_{V}$ and $\mathfrak{X}_{M}$. Suppose now one is interested in studying the association among variables in $L$ within some strict subsets (margins) of $V$ only. An application would be to check collapsibility for the corresponding marginal tables or conditional independence of variables in those margins only.  For this purpose, we define collapsibility by considering two arbitrary marginal tables of $\mathfrak{X}$. Specifically, we consider the marginal tables $\mathfrak{X}_{M}$ and $\mathfrak{X}_{N}$, where $M\subset N\subseteq V$ with $|V|\geq 2$. 
\begin{definition}\label{def2.1}
A $|N|$-dimensional table $\mathfrak{X}_{N}$ is collapsible over $N\backslash M$ into a $|M|$-dimensional table $\mathfrak{X}_{M}$ with respect to $\lambda_{L}^{N}$, where $\emptyset\neq L\subseteq M\subset N$ if
\begin{equation}\label{eq2.2}
\lambda_{L}^{M}(x_{L})=\lambda_{L}^{N}(x_{L}),\quad\forall~x_{L}\in\mathfrak{X}_{L}.
\end{equation}
\end{definition}   
Since $V$ is fixed for a given table and $\lambda_{L}^{N}\neq\lambda_{L}^{V}$ unless $N=V$, Definition \ref{def2.1} differs from that of Bergsma and Rudas (2002), and cannot be obtained by simply relabelling the margins. 
\noindent Define now
\begin{equation}\label{eq2.3}
d^{(M)}(x_{M}) = \log p_{M}(x_{M}) - \nu_{M}^{N}(x_{M})
\end{equation}
and for any $Z\subseteq M$
\begin{equation}\label{eq2.4}
\tilde{d}_{Z}^{(M)}(x_{Z}) = \frac{1}{|\mathfrak{X}_{M\backslash Z}|}\sum_{j_{M}\in\mathfrak{X}_{M}:j_{Z}=x_{Z}}d^{(M)}(j_{M}).
\end{equation}
By definition, $\nu_{M}^{N}(x_{M})$ in (\ref{eq2.3}) is an average of the logarithms of marginal cell probabilities in $\mathfrak{X}_{N}$ over levels of variables in $\mathfrak{X}_{N\backslash M}$. So $d^{(M)}(x_{M})$ represents the difference between the logarithm of a marginal cell probability in $\mathfrak{X}_{M}$ and the quantity $\nu_{M}^{N}(x_{M})$. Also $\tilde{d}_{Z}^{(M)}(x_{Z})$ is the average of the differences $d^{(M)}(x_{M})$ over levels of variables in $\mathfrak{X}_{M\backslash Z}$ for any subset $Z$ of $M$. The following result 
characterizes the conditions for which collapsibility holds with respect to the MLL parameter $\lambda_{L}^{N}$.
\begin{theorem}\label{th2.1}
Let $\emptyset\neq L\subseteq M\subset N$ and $\delta_{Z}=(\lambda_{Z}^{M}-\lambda_{Z}^{N})$ for $Z\subseteq M$. The following conditions are equivalent to collapsibility of  $\mathfrak{X}_{N}$ over $N\backslash M$ into $\mathfrak{X}_{M}$ with respect to $\lambda_{L}^{N}$.
\begin{enumerate}
\item[(i)] $\delta_{L}(x_{L})=0\quad\forall~ x_{L}\in\mathfrak{X}_{L}$; 
\item[(ii)] ${\tilde{d}_{L}}^{(M)}(x_{L})=\sum_{Z\subset L}\delta_{Z}(x_{Z})~\quad\forall~ x_{L}\in\mathfrak{X}_{L},~x_{Z}\in\mathfrak{X}_{Z} $; 
\item[(iii)] $\sum_{Z\subseteq L}(-1)^{|L\backslash Z|} {\tilde{d}_{Z}}^{(M)}(x_{Z})=0~\quad\forall~ x_{Z}\in\mathfrak{X}_{Z}$. 
\end{enumerate}  	
\end{theorem}
Note that Definition \ref{def2.1} gives a condition for checking collapsibility by comparing MLL parameters in $\mathfrak{X}_{M}$ and $\mathfrak{X}_{N}$. However, the necessary and sufficient conditions for collapsibility in Theorem \ref{th2.1} involve simple expressions based on marginal cell probabilities in $\mathfrak{X}_{M}$ , $\mathfrak{X}_{N}$, $\mathfrak{X}_{N\backslash M}$ and $\mathfrak{X}_{M\backslash Z}$ which are easily verifiable. Hence, they provide additional insight into the phenomenon of collapsibility in multidimensional tables.  
\begin{example}\label{ex2.1}
	Consider the 3-dimensional Table \ref{t1} discussed in Whittemore (1978), which cross-classifies variables $X_{1}$, $X_{2}$ and $X_{3}$ having $2$, $2$ and $3$ levels respectively.
	\begin{table}[ht]
		\caption{$2\times 2\times 3$ Table}\label{t1}
		\begin{center}
			$
			\begin{array}{|c|c|c|c|}\hline
			& & X_{2} = 1 & X_{2} = 2 \\ \hline
			X_{3} = 1 & X_{1} = 1 & 75 & 24 \\   
			& X_{1} = 2 & 20  & 16 \\ \hline 
			X_{3} = 2 & X_{1} = 1 & 25 & 8 \\   
			& X_{1} = 2 & 60 & 48 \\ \hline 
			X_{3} = 3 & X_{1} = 1 & 20 & 16 \\   
			& X_{1} = 2 & 16 & 32 \\ \hline 
			\end{array}
			$
		\end{center} 
	\end{table} 
	We have 
	\begin{align*}
		\tilde{d}^{(12)}_{\emptyset}&=1.2356;~	\tilde{d}^{(12)}_{1}(1)=1.2356,~\tilde{d}^{(12)}_{1}(2)=1.2356,~\tilde{d}^{(12)}_{2}(1)=1.2768,~\tilde{d}^{(12)}_{2}(2)=1.1944; \\
		\tilde{d}^{(12)}_{12}(1,1)&=1.2768,~\tilde{d}^{(12)}_{12}(1,2)=1.1944,~\tilde{d}^{(12)}_{12}(2,1)=1.2768,~\tilde{d}^{(12)}_{12}(2,2)=1.1944.
	\end{align*}
	So we obtain
	\begin{align*}
		\tilde{d}^{(12)}_{12}(1,1) - \tilde{d}^{(12)}_{1}(1)&=\tilde{d}^{(12)}_{2}(1) - \tilde{d}^{(12)}_{\emptyset} = 0.0412; \tilde{d}^{(12)}_{12}(1,2) - \tilde{d}^{(12)}_{1}(1) = \tilde{d}^{(12)}_{2}(2) - \tilde{d}^{(12)}_{\emptyset} = -0.0412; \\
		\tilde{d}^{(12)}_{12}(2,1) - \tilde{d}^{(12)}_{1}(2)&=\tilde{d}^{(12)}_{2}(1) - \tilde{d}^{(12)}_{\emptyset} = 0.0412; \tilde{d}^{(12)}_{12}(2,2) - \tilde{d}^{(12)}_{1}(2) = \tilde{d}^{(12)}_{2}(2) - \tilde{d}^{(12)}_{\emptyset} = -0.0412,
	\end{align*}
	which satisfy Condition (iii) of Theorem \ref{th2.1}. Also, we observe that $\hat{\lambda}_{12}^{12}(1,1)=\hat{\lambda}_{12}^{123}(1,1)=\hat{\lambda}_{12}^{12}(2,2)=\hat{\lambda}_{12}^{123}(2,2)=0.2290$ and $\hat{\lambda}_{12}^{12}(1,2)=\hat{\lambda}_{12}^{123}(1,2)=\hat{\lambda}_{12}^{12}(2,1)=\hat{\lambda}_{12}^{123}(2,1)=-0.2290$ implying $\delta_{12}(x_{1},x_{2})=0$ for $x_{1},x_{2}\in\{1,2\}$, which satisfies Condition (i) of Theorem \ref{th2.1}. Finally
	\begin{align*}
		\delta_{\emptyset} + \delta_{1}(1) + \delta_{2}(1) = 1.2768 =  \tilde{d}^{(12)}_{12}(1,1);~
		\delta_{\emptyset} + \delta_{1}(1) + \delta_{2}(2) = 1.1944 =  \tilde{d}^{(12)}_{12}(1,2); \\
		\delta_{\emptyset} + \delta_{1}(2) + \delta_{2}(1) = 1.2768 =  \tilde{d}^{(12)}_{12}(2,1);~
		\delta_{\emptyset} + \delta_{1}(2) + \delta_{2}(2) = 1.1944 =  \tilde{d}^{(12)}_{12}(2,2),
	\end{align*}
	which satisfy Condition (ii) of Theorem \ref{th2.1}. Hence, Table \ref{t1} can be collapsed over $X_{3}$ into a $2\times 2$ table with respect to $\lambda_{12}^{123}$.
\end{example}
As shown in Example 3.2 of Vellaisamy and Vijay (2007), if one is interested in studying the association (say conditional independence) between certain variables with others, then the full table can be collapsed into a marginal table with respect to ordinary log-linear parameters involving those variables as effects. This motivates the case for studying collapsibility with respect to more than one MLL parameter. 
The next result provides a necessary and sufficient condition for collapsibility with respect to a collection of MLL parameters having a set of common effects.
\begin{theorem}\label{th2.2}
Let $\emptyset\neq L\subseteq M\subset N$. A table $\mathfrak{X}_{N}$ is collapsible over $N\backslash M$ into the table $\mathfrak{X}_{M}$ with respect to the set $C_{S}=\{\lambda_{A}^{N}|S\subseteq A\subseteq L\}$ of MLL parameters if and only if
\begin{equation}\label{eq2.9}
\sum_{R\subseteq S}(-1)^{|S\backslash R|}{\tilde{d}}_{L_{R}}^{(M)}(x_{L_{R}})=0,
\end{equation}  
where $L_{R}=L\backslash R$ and $1\leq |S|\leq |M|$.
\end{theorem} 
In Theorem \ref{th2.2}, $S\subseteq M$ denotes the set of common effects for the MLL parameters in the collection $C_{S}$. The necessary and sufficient condition for collapsibility with respect to $C_{S}$ involves a linear combination of the averages $\tilde{d}_{L_{R}}^{(M)}(x_{L_{R}})$ (over variables in $\mathfrak{X}_{M\backslash L_{R}}$) of the differences $d^{(M)}(x_{M})$ (see (\ref{eq2.3}) and (\ref{eq2.4})) for all subsets $R$ of $S$. 
\begin{example}\label{ex2.2}
	Consider the hypothetical $2\times 2\times 3$ Table \ref{t2} similar to the one discussed in Agresti (1990), which cross-classifies $800$ boys according to whether a boy is scout ($X_{1}$), his juvenile delinquency status ($X_{2}$) and his socioeconomic status ($X_{3}$). Here variables $X_{1}$, $X_{2}$ and $X_{3}$ have $2$ (Yes and No), $2$ (Yes and No) and $3$ (Low, Medium and High) levels respectively.
	\begin{table}[ht]
		\caption{$2\times 2\times 3$ Table}\label{t2}
		\begin{center}
			$
			\begin{array}{|c|c|c|c|}\hline
			& & X_{2} = 1 & X_{2} = 2 \\ \hline
			X_{3} = 1 & X_{1} = 1 & 10 & 40 \\   
			& X_{1} = 2 & 40  & 160 \\ \hline 
			X_{3} = 2 & X_{1} = 1 & 18 & 132 \\   
			& X_{1} = 2 & 18 & 132 \\ \hline 
			X_{3} = 3 & X_{1} = 1 & 8 & 192 \\   
			& X_{1} = 2 & 2 & 48 \\ \hline 
			\end{array}
			$
		\end{center} 
	\end{table} 
	From Table \ref{t2}, we have 
	\begin{equation}\label{t2eq1}
		\tilde{d}_{3}^{(13)}(1) = \tilde{d}_{13}^{(13)}(i,1) = 0.9163;~ \tilde{d}_{3}^{(13)}(2) = \tilde{d}_{13}^{(13)}(i,2) = 1.1240;~ \tilde{d}_{3}^{(13)}(3) = \tilde{d}_{13}^{(13)}(i,3) = 1.6298 
	\end{equation}
	for $i\in\{1,2\}$ and 
	\begin{equation}\label{t2eq2}
		\tilde{d}_{3}^{(23)}(1) = \tilde{d}_{23}^{(23)}(j,1) = 0.9163;~ \tilde{d}_{3}^{(23)}(2) = \tilde{d}_{23}^{(23)}(j,2) = 0.6931;~ \tilde{d}_{3}^{(23)}(3) = \tilde{d}_{23}^{(23)}(j,3) = 0.9163
	\end{equation}
	for $j\in\{1,2\}$, which satisfy Condition (\ref{eq2.9}) of Theorem \ref{th2.2}. Hence, Table \ref{t2} is collapsible over $X_{2}$ into a $2\times 3$ table with respect to $\lambda_{1}^{123}$ and $\lambda_{13}^{123}$ using (\ref{t2eq1}). It is also collapsible over $X_{1}$ into a $2\times 3$ table with respect to $\lambda_{2}^{123}$ and $\lambda_{23}^{123}$ using (\ref{t2eq2}). Indeed, it can be shown that $\hat{\lambda}_{1}^{13}(i)=\hat{\lambda}_{1}^{123}(i)=0$ for $i\in\{1,2\}$, $\hat{\lambda}_{13}^{13}(1,1)=\hat{\lambda}_{13}^{123}(1,1)=\hat{\lambda}_{13}^{13}(3,3)=\hat{\lambda}_{13}^{123}(3,3)=-0.6931$, $\hat{\lambda}_{13}^{13}(1,2)=\hat{\lambda}_{13}^{123}(1,2)=\hat{\lambda}_{13}^{13}(2,2)=\hat{\lambda}_{13}^{123}(2,2)=0$ and $\hat{\lambda}_{13}^{13}(1,3)=\hat{\lambda}_{13}^{123}(1,3)=\hat{\lambda}_{13}^{13}(3,1)=\hat{\lambda}_{13}^{123}(3,1)=0.6931$. Also, $\hat{\lambda}_{2}^{23}(1)=\hat{\lambda}_{2}^{123}(1)=-1.0928$, $\hat{\lambda}_{2}^{23}(2)=\hat{\lambda}_{2}^{123}(2)=1.0928$, $\hat{\lambda}_{23}^{13}(1,1)=\hat{\lambda}_{23}^{123}(1,1)=0.3996$, $\hat{\lambda}_{23}^{13}(1,2)=\hat{\lambda}_{23}^{123}(1,2)=0.0966$,
	$\hat{\lambda}_{23}^{13}(1,3)=\hat{\lambda}_{23}^{123}(1,3)=-0.4962$,
	$\hat{\lambda}_{23}^{13}(2,1)=\hat{\lambda}_{23}^{123}(2,1)=-0.3996$, $\hat{\lambda}_{23}^{13}(2,2)=\hat{\lambda}_{23}^{123}(2,2)=-0.0966$ and
	$\hat{\lambda}_{23}^{13}(2,3)=\hat{\lambda}_{23}^{123}(2,3)=0.4962$.
\end{example}

\subsection{Strict Collapsibility of MLL parameters}
We next consider a stronger version of collapsibility, namely, strict collapsibility. For log-linear marginal models, Bergsma and Rudas (2002) mentioned that in addition to $\lambda _{L}^{V}(x_{L}) = \lambda _{L}^{M}(x_{L})$, if $\lambda _{K}^{V}(x_{K}) = 0$ for all $\emptyset\neq L\subset K\not\subseteq M$, then the full table $\mathfrak{X}_{V}$ is said to be {\it strictly collapsible} into the marginal table $\mathfrak{X}_{M}$ with respect to $L$. This implies that the association between the variables in $L$ is the same in both the tables $\mathfrak{X}_{V}$ and $\mathfrak{X}_{M}$ conditionally on any subset of variables not in $M$. 
Similar to Definition \ref{def2.1}, we  provide a definition of strict collapsibility as follows. 
\begin{definition}\label{def3.1}
A $|N|$-dimensional table $\mathfrak{X}_{N}$ is strictly collapsible over $N\backslash M$ into a $|M|$-dimensional table $\mathfrak{X}_{M}$ with respect to $\lambda_{L}^{N}$, where $\emptyset\neq L\subseteq M\subset N$ if
\begin{enumerate}
\item[(i)] $\lambda_{L}^{M}(x_{L}) = \lambda_{L}^{N}(x_{L})\quad\forall~x_{L}\in\mathfrak{X}_{L}$, 
\item[(ii)] $\lambda_{Z}^{N}(x_{Z})=0\quad\forall~L\subset Z\not\subseteq M$, $Z\subseteq N$.		
\end{enumerate}		
\end{definition}
Since $V$ is fixed, note that Definition \ref{def3.1} differs from that of Bergsma and Rudas (2002) unless $N=V$, and cannot be obtained by simply relabelling the margins. 
The result below gives an equivalent expression for Condition (ii) of Definition \ref{def3.1}. 
\begin{lemma}\label{lem3.1}
For a table $\mathfrak{X}_{N}$, we have
\begin{equation*}
\lambda_{Z}^{N}(x_{Z})=0\quad\forall~\emptyset\neq L\subset Z\subseteq N, Z\cap (N\backslash M)\neq\emptyset
\end{equation*} 	
if and only if
\begin{equation}\label{eq3.1}
\sum_{\substack{L\subset Z\subseteq N\\Z\cap (N\backslash M)\neq\emptyset}}\lambda_{Z}^{N}(x_{Z})=0.
\end{equation}
\end{lemma}
Using Lemma \ref{lem3.1}, the following result establishes a necessary and sufficient condition for strict collapsibility with respect to a MLL parameter. 
\begin{theorem}\label{th3.1}
Suppose a table $\mathfrak{X}_{N}$ is collapsible over $N\backslash M$ into the table $\mathfrak{X}_{M}$ with respect to $\lambda_{L}^{N}$, where $\emptyset\neq L\subseteq M\subset N$. Then it is strictly collapsible if and only if
\begin{equation}\label{eq3.5}
\sum_{Z\subseteq L}(-1)^{|L\backslash Z|}\nu_{Z\cup(N\backslash L)}^{N}(x_{Z},x_{N\backslash L}) =\sum_{Z\subseteq L}(-1)^{|L\backslash Z|}\nu_{Z\cup(M\backslash L)}^{N}(x_{Z},x_{M\backslash L}).  
\end{equation}	
\end{theorem}
The necessary and sufficient condition for strict collapsibility in Theorem \ref{th3.1} amounts to checking equality of two linear combinations of the averages $\nu_{Z\cup(N\backslash L)}^{N}$ and $\nu_{Z\cup(M\backslash L)}^{N}$ (see (\ref{eq1.3})) for all subsets $Z$ of $L$.
\begin{example}\label{ex2.3}
	Consider the 3-dimensional Table \ref{t3} discussed in Whittemore (1978), which cross-classifies variables $X_{1}$, $X_{2}$ and $X_{3}$ each having $3$ levels.
	\begin{table}[ht]
		\caption{$3\times 3\times 3$ Table}\label{t3}
		\begin{center}
			$
			\begin{array}{|c|c|c|c|c|}\hline
			& & X_{2} = 1 & X_{2} = 2  & X_{2}=3\\ \hline
			X_{3} = 1 & X_{1} = 1 & 125 & 40 & 75 \\   
			& X_{1} = 2 & 40 & 32  & 24 \\
			& X_{1} = 3 & 75 & 120  & 45 \\ \hline 
			X_{3} = 2 & X_{1} = 1 & 40 & 32 & 120 \\   
			& X_{1} = 2 & 32 & 64 & 96 \\ 
			& X_{1} = 3 & 24 & 96 & 72 \\ \hline 
			X_{3} = 3 & X_{1} = 1 & 75 & 24 & 45 \\   
			& X_{1} = 2 & 120 & 96 & 72 \\  
			& X_{1} = 3 & 45 & 72 & 27 \\ \hline 
			\end{array}
			$
		\end{center} 
	\end{table} 
	We have 
	\begin{align*}
		\nu_{\emptyset}^{123} - \nu_{1}^{123}(1) - \nu_{2}^{123}(1) + \nu_{12}^{123}(1,1) &= \nu_{3}^{123}(k) - \nu_{13}^{123}(1,k) - \nu_{23}^{123}(1,k) + \log p_{123}(1,1,k) = 0.2806 \\
		\nu_{\emptyset}^{123} - \nu_{1}^{123}(1) - \nu_{2}^{123}(2) + \nu_{12}^{123}(1,2) &= \nu_{3}^{123}(k) - \nu_{13}^{123}(1,k) - \nu_{23}^{123}(2,k) + \log p_{123}(1,2,k) = -0.5613; \\
		\nu_{\emptyset}^{123} - \nu_{1}^{123}(1) - \nu_{2}^{123}(3) + \nu_{12}^{123}(1,3) &= \nu_{3}^{123}(k) - \nu_{13}^{123}(1,k) - \nu_{23}^{123}(3,k) + \log p_{123}(1,3,k) = 0.2806; \\
		\nu_{\emptyset}^{123} - \nu_{1}^{123}(2) - \nu_{2}^{123}(1) + \nu_{12}^{123}(2,1) &= \nu_{3}^{123}(k) - \nu_{13}^{123}(2,k) - \nu_{23}^{123}(1,k) + \log p_{123}(2,1,k) = -0.0248; \\
		\nu_{\emptyset}^{123} - \nu_{1}^{123}(2) - \nu_{2}^{123}(2) + \nu_{12}^{123}(2,2) &= \nu_{3}^{123}(k) - \nu_{13}^{123}(2,k) - \nu_{23}^{123}(2,k) + \log p_{123}(2,2,k) = 0.0496; \\
		\nu_{\emptyset}^{123} - \nu_{1}^{123}(2) - \nu_{2}^{123}(3) + \nu_{12}^{123}(2,3) &= \nu_{3}^{123}(k) - \nu_{13}^{123}(2,k) - \nu_{23}^{123}(3,k) + \log p_{123}(2,3,k) = -0.0248; \\
		\nu_{\emptyset}^{123} - \nu_{1}^{123}(3) - \nu_{2}^{123}(1) + \nu_{12}^{123}(3,1) &= \nu_{3}^{123}(k) - \nu_{13}^{123}(3,k) - \nu_{23}^{123}(1,k) + \log p_{123}(3,1,k) = -0.2558; \\
		\nu_{\emptyset}^{123} - \nu_{1}^{123}(3) - \nu_{2}^{123}(2) + \nu_{12}^{123}(3,2) &= \nu_{3}^{123}(k) - \nu_{13}^{123}(3,k) - \nu_{23}^{123}(2,k) + \log p_{123}(3,2,k) = 0.5117; \\
		\nu_{\emptyset}^{123} - \nu_{1}^{123}(3) - \nu_{2}^{123}(3) + \nu_{12}^{123}(3,3) &= \nu_{3}^{123}(k) - \nu_{13}^{123}(3,k) - \nu_{23}^{123}(3,k) + \log p_{123}(3,3,k) = -0.2558,
	\end{align*}
	for $k\in\{1,2,3\}$, which maybe written as
	\begin{equation*}
		\nu_{\emptyset}^{123} - \nu_{1}^{123}(i) - \nu_{2}^{123}(j) + \nu_{12}^{123}(i,j) = \nu_{3}^{123}(k) - \nu_{13}^{123}(i,k) - \nu_{23}^{123}(j,k) + \log p_{123}(i,j,k)
	\end{equation*}
	for all $i,j,k\in\{1,2,3\}$. Similarly, it can be shown that
	\begin{align*}
		\nu_{\emptyset}^{123} - \nu_{1}^{123}(i) - \nu_{3}^{123}(k) + \nu_{13}^{123}(i,k) = \nu_{2}^{123}(j) - \nu_{12}^{123}(i,j) - \nu_{23}^{123}(j,k) + \log p_{123}(i,j,k); \\
		\nu_{\emptyset}^{123} - \nu_{2}^{123}(j) - \nu_{3}^{123}(k) + \nu_{23}^{123}(j,k) = \nu_{1}^{123}(i) - \nu_{12}^{123}(i,j) - \nu_{13}^{123}(i,k) + \log p_{123}(i,j,k).
	\end{align*}
	Note that each of the above equations satisfies Condition (\ref{eq3.5}) of Theorem \ref{th3.1}. Hence, Table \ref{t3} is strictly collapsible over $X_{1}$, $X_{2}$ and $X_{3}$ with respect to $\lambda_{23}^{123}$, $\lambda_{13}^{123}$ and $\lambda_{12}^{123}$ respectively into $3\times 3$ tables. Indeed, we can show $\hat{\lambda}_{12}^{12}(i,j)=\hat{\lambda}_{12}^{123}(i,j)$, $\hat{\lambda}_{13}^{13}(i,k)=\hat{\lambda}_{13}^{123}(i,k)$ and $\hat{\lambda}_{23}^{23}(j,k)=\hat{\lambda}_{23}^{123}(j,k)$ alongwith $\hat{\lambda}_{123}^{123}(i,j,k)=0$ for all $i,j,k\in\{1,2,3\}$.
\end{example}
Next, we provide necessary and sufficient conditions for strict collapsibility with respect to a collection of MLL parameters having a set of common effects.
\begin{theorem}\label{th3.2}
A table $\mathfrak{X}_{N}$ is strictly collapsible over $N\backslash M$ into a table $\mathfrak{X}_{M}$ with respect to the set $C_{S}= \{\lambda_{A}^{N}|S\subseteq A\subseteq L\}$ of interaction parameters if and only if
\begin{equation}\label{eq3.11}
\sum_{R\subseteq S}(-1)^{|S\backslash R|}{\tilde{d}}_{L_{R}}^{(M)}(x_{L_{R}})=0
\end{equation} 	
and 
\begin{equation}\label{eq3.12}
\sum_{Z\subseteq S}(-1)^{|S\backslash Z|}\nu_{N_{S}\cup Z}^{N}(x_{N_{S}},x_{Z}) = \sum_{Z\subseteq S}(-1)^{|S\backslash Z|}\nu_{M_{S}\cup Z}^{N}(x_{M_{S}},x_{Z}),
\end{equation}
where $M_{S}=M\backslash S$, $N_{S}=N\backslash S$ and $1\leq |S|\leq |M|$.	
\end{theorem}
From Theorem \ref{th3.2} observe that in addition to the colapsibility condition in Theorem \ref{th2.2}, we need to check equality of the linear combinations of the averages $\nu_{N_{S}\cup Z}^{N}$ and $\nu_{M_{S}\cup Z}^{N}$ for all subsets $Z$ of $S$ where $S\subseteq M$ is the set of common effects corresponding to the collection $C_{S}$ of MLL parameters.
\begin{example}\label{ex2.4}
	Consider Table \ref{t2} in Example \ref{ex2.2}. For  $j\in\{1,2\}$, we obtain
	\begin{align}\label{t4eq1}
		\log p_{123}(1,j,1) - \nu_{23}^{123}(j,1) &= \nu_{13}^{123}(1,1) - \nu_{3}^{123}(1) = -0.6931;\nonumber \\
		\log p_{123}(1,j,2) - \nu_{23}^{123}(j,2) &= \nu_{13}^{123}(1,2) - \nu_{3}^{123}(2) = 0;\nonumber \\
		\log p_{123}(1,j,3) - \nu_{23}^{123}(j,3) &= \nu_{13}^{123}(1,3) - \nu_{3}^{123}(3) = 0.6931; \nonumber \\
		\log p_{123}(2,j,1) - \nu_{23}^{123}(j,1) &= \nu_{13}^{123}(2,1) - \nu_{3}^{123}(1) = 0.6931; \nonumber \\
		\log p_{123}(2,j,2) - \nu_{23}^{123}(j,2) &= \nu_{13}^{123}(2,2) - \nu_{3}^{123}(2) = 0; \nonumber\\
		\log p_{123}(2,j,3) - \nu_{23}^{123}(j,3) &= \nu_{13}^{123}(2,3) - \nu_{3}^{123}(3) = -0.6931
	\end{align}
	and for $i\in\{1,2\}$, we have
	\begin{align}\label{t4eq2}
		\log p_{123}(i,1,1) - \nu_{13}^{123}(i,1) &= \nu_{23}^{123}(1,1) - \nu_{3}^{123}(1) = -0.6931; \nonumber \\
		\log p_{123}(i,1,2) - \nu_{13}^{123}(i,2) &= \nu_{23}^{123}(1,2) - \nu_{3}^{123}(2) = -0.9962; \nonumber \\
		\log p_{123}(i,1,3) - \nu_{13}^{123}(i,3) &= \nu_{23}^{123}(1,3) - \nu_{3}^{123}(3) = -1.5890 \nonumber \\
		\log p_{123}(i,2,1) - \nu_{13}^{123}(i,1) &= \nu_{23}^{123}(2,1) - \nu_{3}^{123}(1) = 0.6931; 
		\nonumber \\
		\log p_{123}(i,2,2) - \nu_{13}^{123}(i,2) &= \nu_{23}^{123}(2,2) - \nu_{3}^{123}(2) = 0.9962;
		\nonumber \\ 
		\log p_{123}(i,2,3) - \nu_{13}^{123}(i,3) &= \nu_{23}^{123}(2,3) - \nu_{3}^{123}(3) = 1.5890,
	\end{align}
	which satisfy Condition (\ref{eq3.12}) of Theorem \ref{th3.2}. As shown in Example \ref{ex2.2}, Table \ref{t2} also satisfies Condition (\ref{eq3.11}) of Theorem \ref{th3.2}. Hence, Table \ref{t2} is strictly collapsible over $X_{2}$ into a $2\times 3$ table with respect to $\lambda_{1}^{123}$ and $\lambda_{13}^{123}$ using (\ref{t4eq1}). It is also strictly collapsible over $X_{1}$ into a $2\times 3$ table with respect to $\lambda_{2}^{123}$ and $\lambda_{23}^{123}$ using (\ref{t4eq2}). This coincides with our observation in Example \ref{ex2.2} that $\hat{\lambda}_{1}^{13}(i)=\hat{\lambda}_{1}^{123}(i)$, $\hat{\lambda}_{13}^{13}(i,k)=\hat{\lambda}_{13}^{123}(i,k)$, $\hat{\lambda}_{2}^{23}(j)=\hat{\lambda}_{2}^{123}(j)$ and $\hat{\lambda}_{23}^{23}(j,k)=\hat{\lambda}_{23}^{123}(j,k)$ for $i,j\in\{1,2\}$ and $k\in\{1,2,3\}$. Indeed, it can also be shown that $\hat{\lambda}_{123}^{123}(i,j,k)=0$ for $i,j,k\in\{1,2,3\}$. 
\end{example}

\section{Strict Collapsibility and Independence}
In this section, we study the relationship between strict collapsibility and various forms of independence -- conditional, joint and mutual among variables in a multidimensional table. Suppose $A$, $B$ and $C$ form a partition of $M$ (a finite index set). We have 
\begin{enumerate}
\item[(i)] $X_{A}\ci X_{B}|X_{C}$ (conditional independence) if and only if
\begin{align*}
p_{M}(x_{M}) &= \frac{p_{AC}(x_{AC})p_{BC}(x_{BC})}{p_{C}(x_{C})}, 
\end{align*}
\item[(ii)] $X_{A}\ci (X_{B},X_{C})$ (joint independence) if and only if
\begin{align*}
p_{M}(x_{M}) &=p_{A}(x_{A})p_{BC}(x_{BC}),
\end{align*}
\item[(iii)] $X_{A}\ci X_{B}\ci X_{C}$ (mutual independence) if and only if 
\begin{align*}
p_{M}(x_{M}) &= p_{A}(x_{A})p_{B}(x_{B})p_{C}(x_{C}) .
\end{align*}
\end{enumerate}
Simpson (1951) had proved necessity and sufficiency of conditional independence for collapsibility in a $2\times 2\times 2$ table. For an arbitrary 3-dimensional table, Theorem 2.4-1 of Bishop, Fienberg and Holland (1975) states that conditional independence is a necessary and sufficient condition for collapsibility of ordinary log-linear parameters, while Theorem 2.5-1 states the same result for a multidimensional table. However, Whittemore (1978) showed through counterexamples that the above theorems are not true (non-necessity of conditional independence) for collapsibility in such tables. It was also shown that vanishing of the 3-factor interaction is neither necessary nor sufficient for collapsibility in a 3-dimensional table. 

We provide a result below stating necessary and sufficient conditions for strict collapsibility of MLL parameters (in a multidimensional table) in terms of conditional independence.
\begin{theorem}\label{th4.1}
A $|M|$-dimensional table $\mathfrak{X}_{M}$ is strictly collapsible over $\mathfrak{X}_{A}$ ($\mathfrak{X}_{B}$) into $\mathfrak{X}_{BC}$ ($\mathfrak{X}_{AC}$) with respect to $\lambda_{B'}^{ABC}$ and $\lambda_{B'C'}^{ABC}$ ($\lambda_{A'}^{ABC}$ and $\lambda_{A'C'}^{ABC}$) if and only if $X_{A}\ci X_{B}|X_{C}$, where $A'\subseteq A$, $B'\subseteq B$, $C'\subseteq C$ and $\{A,B, C\}$ is a partition of $M$. 
\end{theorem}
\begin{remark}\label{rem4.2}
Conditional indpendence is always sufficient for strict collapsibility in a $|M|$-dimensional table. However, it is not necessary for the particular case when $|A|=|B|=|C|=1$ (a 3-dimensional table) and $\mathfrak{X}_{M}$ is strictly collapsible with respect to two-way interaction parameters only ($\lambda_{A'C'}^{ABC}$ or $\lambda_{B'C'}^{ABC}$). Then $\lambda_{A'B'C'}^{ABC}= 0$, while $\lambda_{A'B'}^{ABC}$ is non-zero implying $X_{A}\not\ci X_{B}|X_{C}$. This observation is consistent with Theorem 1 of Whittemore (1978), which states the existence of arbitrary 3-dimensional tables that are strictly collapsible over each variable such that no two-way log-linear interaction vanishes.
\end{remark}

\begin{corollary}\label{cor4.1}	 
A 3-dimensional table $\mathfrak{X}_{pqr}$ is strictly collapsible over $\mathfrak{X}_{p}$ ($\mathfrak{X}_{q}$) into a 2-dimensional table $\mathfrak{X}_{qr}$ ($\mathfrak{X}_{pr}$) with respect to $\lambda_{q}^{123}$ and $\lambda_{qr}^{123}$ ($\lambda_{p}^{123}$ and $\lambda_{pr}^{123}$) if and only if $X_{p}\ci X_{q}|X_{r}$ where $p\neq q\neq r\in\{1,2,3\}$.
\end{corollary}

\begin{example}\label{ex3.1}
	Consider the $2\times 2\times 2$ Table \ref{t4}, which deals with the job satisfaction data of 715 blue collar workers  from a large scale investigation into the Danish industry in 1968 (see Andersen (1990)). The three variables are supervisor satisfaction $X_{1}$ having 2 levels (low and high), worker job satisfaction ($X_{2}$) having 2 levels (low and high) and quality of management ($X_{3}$) having $2$ levels (bad and good).
	\begin{table}[ht]
		\caption{$2\times 2\times 2$ Table}\label{t4}
		\begin{center}
			$
			\begin{array}{|c|c|c|c|}\hline
			& & X_{2} = 1 & X_{2} = 2 \\ \hline
			X_{3} = 1 & X_{1} = 1 & 103 & 87 \\   
			& X_{1} = 2 & 32  & 42 \\ \hline 
			X_{3} = 2 & X_{1} = 1 & 59 & 109 \\   
			& X_{1} = 2 & 78 & 205 \\ \hline 
			\end{array}
			$
		\end{center} 
	\end{table} 
	For testing the goodness of fit of various alternative models against the null saturated model, we fit several hierarchical log-linear marginal models to Table \ref{t4} including models of conditional independence, of joint independence and of mutual independence alongwith the no three-factor interaction model. A log-linear marginal model satisfying (\ref{eq1.11a}) in Lemma \ref{lem1.11} is said to be hierarchical if $\lambda_{L'}^{M}=0\Rightarrow\lambda_{L}^{M}=0$ for $L'\subset L$. The models with adequate fit for Table \ref{t4} are the conditional independence model $X_{1}\ci X_{2}|X_{3}$ (given by $\lambda_{12}^{123}=\lambda_{123}^{123}=0$) and the no three-factor interaction model (given by $\lambda_{123}^{123}=0$) based on $p$-values of $0.0676$ and $0.7988$ respectively. However, the most parsimonious model is the conditional independence model $X_{1}\ci X_{2}|X_{3}$, that is, supervisor satisfaction does not depend on the quality of management given worker satisfaction. By Theorem \ref{th4.1}, Table \ref{t4} should be strictly collapsible over $X_{1}$ ($X_{2}$) with respect to $\lambda_{2}^{123}$ and $\lambda_{23}^{123}$ ($\lambda_{1}^{123}$ and $\lambda_{13}^{123}$). We verify this next. Table \ref{t5} shows the table of expected values obtained under the model $X_{1}\ci X_{2}|X_{3}$.
	\begin{table}[ht]
		\caption{$2\times 2\times 2$ Table}\label{t5}
		\begin{center}
			$
			\begin{array}{|c|c|c|c|}\hline
			& & X_{2} = 1 & X_{2} = 2 \\ \hline
			X_{3} = 1 & X_{1} = 1 & 97.1591 & 92.8409 \\   
			& X_{1} = 2 & 37.8409  & 36.1591 \\ \hline 
			X_{3} = 2 & X_{1} = 1 & 51.0333 & 116.9667 \\   
			& X_{1} = 2 & 85.9667 & 197.0333 \\ \hline 
			\end{array}
			$
		\end{center} 
	\end{table} 
	From Table \ref{t5}, we have for $i\in\{1,2\}$ 
	\begin{equation}\label{t5eq1}
		\tilde{d}_{3}^{(13)}(1) = \tilde{d}_{13}^{(13)}(i,1) = 0.6934;~ \tilde{d}_{3}^{(13)}(2) = \tilde{d}_{13}^{(13)}(i,2) = 0.7768; 
	\end{equation}
	and for $j\in\{1,2\}$ 
	\begin{equation}\label{t5eq2}
		\tilde{d}_{3}^{(23)}(1) = \tilde{d}_{23}^{(23)}(j,1) = 0.8004;~ \tilde{d}_{3}^{(23)}(2) = \tilde{d}_{13}^{(23)}(j,2) = 0.7268, 
	\end{equation}
	which satisfy Condition (\ref{eq3.11}) of Theorem \ref{th3.2}. Again from Table \ref{t5}, we obtain for $j\in\{1,2\}$
	\begin{align}\label{t5eq3}
		\log p_{123}(1,j,1) - \nu_{23}^{123}(j,1) &= \nu_{13}^{123}(1,1) - \nu_{3}^{123}(1) = 0.4715;\nonumber \\
		\log p_{123}(1,j,2) - \nu_{23}^{123}(j,2) &= \nu_{13}^{123}(1,2) - \nu_{3}^{123}(2) = -0.2607;\nonumber \\
		\log p_{123}(2,j,1) - \nu_{23}^{123}(j,1) &= \nu_{13}^{123}(2,1) - \nu_{3}^{123}(1) = -0.4715; \nonumber \\
		\log p_{123}(2,j,2) - \nu_{23}^{123}(j,2) &= \nu_{13}^{123}(2,2) - \nu_{3}^{123}(2) = 0.2607
	\end{align}
	and for $i\in\{1,2\}$
	\begin{align}\label{t5eq4}
		\log p_{123}(i,1,1) - \nu_{13}^{123}(i,1) &= \nu_{23}^{123}(1,1) - \nu_{3}^{123}(1) = 0.0227; \nonumber \\
		\log p_{123}(i,1,2) - \nu_{13}^{123}(i,2) &= \nu_{23}^{123}(1,2) - \nu_{3}^{123}(2) = -0.4147; \nonumber \\
		\log p_{123}(i,2,1) - \nu_{13}^{123}(i,1) &= \nu_{23}^{123}(2,1) - \nu_{3}^{123}(1) = -0.0227; 
		\nonumber \\
		\log p_{123}(i,2,2) - \nu_{13}^{123}(i,2) &= \nu_{23}^{123}(2,2) - \nu_{3}^{123}(2) = 0.4147,
	\end{align}
	which satisfy Condition (\ref{eq3.12}) of Theorem \ref{th3.2}. Hence, Table \ref{t5} is strictly collapsible over $X_{2}$ into a $2\times 3$ table with respect to $\lambda_{1}^{123}$ and $\lambda_{13}^{123}$ using (\ref{t5eq1}) and (\ref{t5eq3}). It is also strictly collapsible over $X_{1}$ into a $2\times 3$ table with respect to $\lambda_{2}^{123}$ and $\lambda_{23}^{123}$ using (\ref{t5eq2}) and (\ref{t5eq4}). Indeed, it can be shown $\hat{\lambda}_{1}^{13}(1) = \hat{\lambda}_{1}^{123}(1) = 0.1054$, $\hat{\lambda}_{1}^{13}(2) = \hat{\lambda}_{1}^{123}(2) = -0.1054$, $\hat{\lambda}_{13}^{13}(1,1) = \hat{\lambda}_{13}^{123}(1,1) = \hat{\lambda}_{13}^{13}(2,2) = \hat{\lambda}_{13}^{123}(2,2) = 0.3661$ and $\hat{\lambda}_{13}^{13}(1,2) = \hat{\lambda}_{13}^{123}(1,2) = \hat{\lambda}_{13}^{13}(2,1) = \hat{\lambda}_{13}^{123}(2,1) = -0.3661$ alongwith $\hat{\lambda}_{123}^{123}(i,j,k)=0$ for $i,j,k\in\{1,2\}$. Also, $\hat{\lambda}_{2}^{23}(1) = \hat{\lambda}_{2}^{123}(1) = -0.1960$, $\hat{\lambda}_{2}^{23}(2) = \hat{\lambda}_{2}^{123}(2) = 0.1960$, $\hat{\lambda}_{23}^{23}(1,1) = \hat{\lambda}_{23}^{123}(1,1) = \hat{\lambda}_{23}^{23}(2,2) = \hat{\lambda}_{23}^{123}(2,2) = 0.2187$ and $\hat{\lambda}_{23}^{23}(1,2) = \hat{\lambda}_{23}^{123}(1,2) = \hat{\lambda}_{23}^{23}(2,1) = \hat{\lambda}_{23}^{123}(2,1) = -0.2187$ alongwith $\hat{\lambda}_{123}^{123}(i,j,k)=0$ for $i,j,k\in\{1,2\}$.
\end{example}

The next result shows the relationship between strict collapsibility of MLL parameters and joint independence in a multidimensional table.	
\begin{theorem}\label{th4.2}
A $|M|$-dimensional table $\mathfrak{X}_{M}$ is strictly collapsible with respect to $\lambda_{B'}^{ABC}$, $\lambda_{C'}^{ABC}$ and $\lambda_{B'C'}^{ABC}$ by collapsing over $\mathfrak{X}_{A}$ into $\mathfrak{X}_{BC}$ if and only if $X_{A}\ci (X_{B},X_{C})$, where $A'\subseteq A$, $B'\subseteq B$, $C'\subseteq C$ and $\{A,B,C\}$ is a partition of $M$.
\end{theorem}
\begin{remark}\label{rem4.21}
Joint independence is always sufficient for strict collapsibility in a $|M|$-dimensional table. However, it is not necessary for the specific case when $|A|=|B|=|C|=1$ (a 3-dimensional table) and $\mathfrak{X}_{M}$ is strictly collapsible over $\mathfrak{X}_{A}$ with respect to $\lambda_{B'C'}^{ABC}$ only. Then we have $\lambda_{A'B'C'}^{ABC}=0$ while $\lambda_{A'B'}^{ABC}$ and $\lambda_{A'C'}^{ABC}$ are non-zero implying $X_{A}\not\ci(X_{B},X_{C})$.
\end{remark}

\begin{corollary}\label{cor4.2}
A 3-dimensional table $\mathfrak{X}_{pqr}$, where $p\neq q\neq r\in\{1,2,3\}$, is strictly collapsible with respect to $\lambda_{q}^{123}$, $\lambda_{r}^{123}$ and $\lambda_{qr}^{123}$ by collapsing over $\mathfrak{X}_{p}$ into $\mathfrak{X}_{qr}$
if and only if $X_{p}\ci (X_{q},X_{r})$. Similarly, $\mathfrak{X}_{pqr}$ is strictly collapsible with respect to $\lambda_{p}^{123}$ by collapsing over $\mathfrak{X}_{q}$ ($\mathfrak{X}_{r}$) into $\mathfrak{X}_{pr}$ ($\mathfrak{X}_{pq}$) if and only if $X_{p}\ci (X_{q},X_{r})$. 
\end{corollary}

\begin{example}\label{ex3.2}
	Consider the $2\times 2\times 3$ Table \ref{t6}, which concerns classroom behaviour (see Everitt (1977)) of 97 students classified according to three factors: \\
	1. $X_{1}$ - Teacher’s rating of classroom behaviour (behaviour) with levels `non deviant' (1) and `deviant' (2), \\
	2. $X_{2}$ - Risk index based on home conditions (risk) with levels `not at risk' (1) and `at risk' (2), and \\
	3. $X_{3}$ - Adversity of school conditions (adversity) with levels `low' (1), `medium' (2) and `high' (3). \\
	\begin{table}[ht]
		\caption{$2\times 2\times 3$ Table}\label{t6}
		\begin{center}
			$
			\begin{array}{|c|c|c|c|}\hline
			& & X_{2} = 1 & X_{2} = 2 \\ \hline
			X_{3} = 1 & X_{1} = 1 & 16 & 7 \\   
			& X_{1} = 2 & 1  & 1 \\ \hline 
			X_{3} = 2 & X_{1} = 1 & 15 & 34 \\   
			& X_{1} = 2 & 3 & 8 \\ \hline 
			X_{3} = 3 & X_{1} = 1 & 5 & 3 \\   
			& X_{1} = 2 & 1 & 3 \\ \hline 
			\end{array}
			$
		\end{center} 
	\end{table} 
	For testing the goodness of fit of various alternative models against the null saturated model, we fit several hierarchical log-linear marginal models to Table \ref{t6} including models of conditional independence, of joint independence and of mutual independence alongwith the no three-factor interaction model. The plausible models for Table \ref{t6} are the conditional independence model $X_{1}\ci X_{3}|X_{2}$ (given by $\lambda_{13}^{123}=\lambda_{123}^{123}=0$), the conditional independence model $X_{1}\ci X_{2}|X_{3}$ (given by $\lambda_{12}^{123}=\lambda_{123}^{123}=0$), the joint independence model $X_{1}\ci (X_{2},X_{3})$ (given by $\lambda_{12}^{123}=\lambda_{13}^{123}=\lambda_{123}^{123}=0$) and the no three-factor interaction model (given by $\lambda_{123}^{123}=0$) based on $p$-values of $0.3903$, $0.5926$, $0.3514$ and $0.6241$ respectively. However, the most parsimonious model is the joint independence model $X_{1}\ci (X_{2},X_{3})$, that is, teacher's rating of classroom behaviour does not depend on the factors of risk index based on home conditions and adversity of school conditions jointly. By Theorem \ref{th4.2}, Table \ref{t6} should be strictly collapsible over $X_{1}$ with respect to $\lambda_{2}^{123}$, $\lambda_{3}^{123}$ and $\lambda_{23}^{123}$. We verify this next. Table \ref{t7} shows the table of expected values obtained under the model $X_{1}\ci (X_{2},X_{3})$.
	\begin{table}[ht]
		\caption{$2\times 2\times 3$ Table}\label{t7}
		\begin{center}
			$
			\begin{array}{|c|c|c|c|}\hline
			& & X_{2} = 1 & X_{2} = 2 \\ \hline
			X_{3} = 1 & X_{1} = 1 & 14.0206 & 6.5979 \\   
			& X_{1} = 2 & 2.9794  & 1.4021 \\ \hline 
			X_{3} = 2 & X_{1} = 1 & 14.8454 & 34.6392 \\   
			& X_{1} = 2 & 3.1546 & 7.3608 \\ \hline 
			X_{3} = 3 & X_{1} = 1 & 4.9484 & 4.9484 \\   
			& X_{1} = 2 & 1.0516 & 1.0516 \\ \hline 
			\end{array}
			$
		\end{center} 
	\end{table}
	From Table \ref{t7}, we have for $j\in\{1,2\}$ and $k\in\{1,2,3\}$
	\begin{equation}\label{t7eq1}
		\tilde{d}_{2}^{(23)}(j) =  \tilde{d}_{3}^{(23)}(k) = \tilde{d}_{23}^{(23)}(j,k) = 0.9671,
	\end{equation}
	which satisfies Condition (\ref{eq3.11}) of Theorem \ref{th3.2}. Again from Table \ref{t7}, we obtain for $i\in\{1,2\}$
	\begin{align}\label{t7eq2}
		\log p_{123}(i,1,1) - \nu_{13}^{123}(i,1) &= \nu_{23}^{123}(1,1) - \nu_{3}^{123}(1) = 0.3769;\nonumber \\
		\log p_{123}(i,1,2) - \nu_{13}^{123}(i,2) &= \nu_{23}^{123}(1,2) - \nu_{3}^{123}(2) = -0.4236;\nonumber \\
		\log p_{123}(i,1,3) - \nu_{13}^{123}(i,3) &= \nu_{23}^{123}(1,3) - \nu_{3}^{123}(3) = 0; \nonumber \\
		\log p_{123}(i,2,1) - \nu_{13}^{123}(i,1) &= \nu_{23}^{123}(2,1) - \nu_{3}^{123}(1) = -0.3769; \nonumber \\
		\log p_{123}(i,2,2) - \nu_{13}^{123}(i,2) &= \nu_{23}^{123}(2,2) - \nu_{3}^{123}(2) = 0.4236; \nonumber\\
		\log p_{123}(i,2,3) - \nu_{13}^{123}(i,3) &= \nu_{23}^{123}(2,3) - \nu_{3}^{123}(3) = 0
	\end{align}
	and 
	\begin{align}\label{t7eq3}
		\log p_{123}(i,1,1) - \nu_{12}^{123}(i,1) &= \nu_{23}^{123}(1,1) - \nu_{2}^{123}(1) = 0.3281; \nonumber \\
		\log p_{123}(i,1,2) - \nu_{12}^{123}(i,1) &= \nu_{23}^{123}(1,2) - \nu_{2}^{123}(2) = 0.3853; \nonumber \\
		\log p_{123}(i,1,3) - \nu_{12}^{123}(i,1) &= \nu_{23}^{123}(1,3) - \nu_{2}^{123}(3) = -0.7134; \nonumber \\
		\log p_{123}(i,2,1) - \nu_{12}^{123}(i,2) &= \nu_{23}^{123}(2,1) - \nu_{2}^{123}(1) =  -0.4568; 
		\nonumber \\
		\log p_{123}(i,2,2) - \nu_{12}^{123}(i,2) &= \nu_{23}^{123}(2,2) - \nu_{2}^{123}(2) = 1.2014;
		\nonumber \\ 
		\log p_{123}(i,2,3) - \nu_{12}^{123}(i,2) &= \nu_{23}^{123}(2,3) - \nu_{2}^{123}(3) = -0.7445,
	\end{align}
	which satisfy Condition (\ref{eq3.12}) of Theorem \ref{th3.2}. Hence, Table \ref{t7} is strictly collapsible over $X_{1}$ into a $2\times 3$ table with respect to $\lambda_{2}^{123}$, $\lambda_{3}^{123}$ and $\lambda_{23}^{123}$ using (\ref{t7eq1})-(\ref{t7eq3}). Indeed, it can be shown $\hat{\lambda}_{2}^{23}(1) = \hat{\lambda}_{2}^{123}(1) = -0.0156$, $\hat{\lambda}_{2}^{23}(2) = \hat{\lambda}_{2}^{123}(2) = 0.0156$, $\hat{\lambda}_{3}^{23}(1) = \hat{\lambda}_{3}^{123}(1) = -0.0644$, $\hat{\lambda}_{3}^{23}(2) = \hat{\lambda}_{3}^{123}(2) =  0.7933$, $\hat{\lambda}_{3}^{23}(3) = \hat{\lambda}_{3}^{123}(3) = -0.7289$, $\hat{\lambda}_{23}^{123}(1,1)=0.3925$, $\hat{\lambda}_{23}^{123}(1,2)=-0.4081$, $\hat{\lambda}_{23}^{123}(1,3)=0.0156$, $\hat{\lambda}_{23}^{123}(2,1)=-0.3925$, $\hat{\lambda}_{23}^{123}(2,2)=0.4081$ and $\hat{\lambda}_{23}^{123}(2,3)=-0.0156$ alongwith $\hat{\lambda}_{123}^{123}(i,j,k)=0$ for $i,j\in\{1,2\}$ and $k\in\{1,2,3\}$.	
\end{example}

The following result gives the connection between strict collapsibility of MLL parameters and mutual independence in a multidimensional table.
\begin{theorem}\label{th4.3}
A $|M|$-dimensional table $\mathfrak{X}_{M}$ where $A$, $B$ and $C$ form a partition of $M$, is strictly collapsible with respect to any two of the following sets of MLL parameters:
\begin{enumerate}
\item[1.] $\lambda_{B'}^{ABC}$ and $\lambda_{C'}^{ABC}$ by collapsing over $\mathfrak{X}_{A}$ into $\mathfrak{X}_{BC}$, 
\item[2.] $\lambda_{A'}^{ABC}$ and $\lambda_{C'}^{ABC}$ by collapsing over $\mathfrak{X}_{B}$ into $\mathfrak{X}_{AC}$, 
\item[3.] $\lambda_{A'}^{ABC}$ and $\lambda_{B'}^{ABC}$ by collapsing over $\mathfrak{X}_{C}$ into $\mathfrak{X}_{AB}$
\end{enumerate}	
if and only if $X_{A}\ci X_{B}\ci X_{C}$, where $A'\subseteq A$, $B'\subseteq B$ and $C'\subseteq C$.
\end{theorem}

\begin{remark}\label{rem4.4}
Mutual independence is always necessary and sufficient for strict collapsibility in a $|M|$-dimensional table.
\end{remark}

\begin{corollary}\label{cor4.3}
A 3-dimensional table $\mathfrak{X}_{pqr}$, where $p\neq q\neq r\in\{1,2,3\}$, is strictly collapsible with respect to any two of the following sets of MLL parameters:
\begin{enumerate}
\item[1.] $\lambda_{q}^{123}$ and $\lambda_{r}^{123}$ by collapsing over $\mathfrak{X}_{p}$ into $\mathfrak{X}_{qr}$, 
\item[2.] $\lambda_{p}^{123}$ and $\lambda_{r}^{123}$ by collapsing over $\mathfrak{X}_{q}$ into $\mathfrak{X}_{pr}$, 
\item[3.] $\lambda_{p}^{123}$ and $\lambda_{q}^{123}$ by collapsing over $\mathfrak{X}_{r}$ into $\mathfrak{X}_{pq}$
\end{enumerate}	
if and only if $X_{p}\ci X_{q}\ci X_{r}$.
\end{corollary}

\section{Smoothness of marginal log-linear parameters under collapsibility}
In this section, we explore the relationship between smooth parameterization and collapsibility in the context of a log-linear marginal model. We first prove a general result on MLL parameters defined within different margins but having a common effect. Then a sufficient condition is provided to show the existence of a smooth MLL parameterization under collapsibility conditions, which is the main result of this section. We also establish a sufficient condition for collapsibility in terms of conditional independence of the variables.

The MLL parameters $\{\lambda_{L}^{M}|L\subseteq M\}$ parameterize a marginal distribution $p_{M}$. Similarly, the conditional distribution $X_{A}|X_{B}$ for disjoint $A$ and $B$ can be smoothly parameterized (see Evans (2015)) by 
\begin{equation}\label{eq5.1}
\lambda_{A|B}\equiv\{\lambda_{L}^{AB}\mid L\subseteq AB,~L\cap A\neq\emptyset\}.
\end{equation} 
That is, $\lambda_{A|B}$ is the collection of all MLL parameters defined within the margin $AB$, whose effects contain some element of $A$. For $L\subseteq M\subset N$, Theorem 3.1 of Evans (2015) provides the exact relationship between $\lambda_{L}^{M}$ and $\lambda_{L}^{N}$ for the case of binary variables. We extend their result to the case of general categorical variables with arbitrary number of levels as follows.

\begin{theorem}\label{th5.1}
Let $A$ and $B$ be disjoint subsets of $V$ with $|\mathfrak{X}_{v}|\geq 2$ for $v\in V$. Then the MLL parameter $\lambda_{L}^{AB}$ may be decomposed as
\begin{equation}\label{eq5.2}
\lambda_{L}^{AB} = \lambda_{L}^{B} + f(\lambda_{A|B})
\end{equation}  
for a smooth function $f$ which vanishes if $X_{A}\ci X_{v}|X_{B\backslash\{v\}}$ for some $v\in L$.  
\end{theorem}
Theorem 3 of Bergsma and Rudas (2002) showed that for $L\subseteq M\subset N$, the MLL parameters $\lambda_{L}^{M}$ and $\lambda_{L}^{N}$ are linearly dependent at certain points in the parameter space. Hence, no smooth parameterization can include two such parameters. As a result, collapsibility conditions (see (\ref{eq2.2})) generally do not define a curved exponential family. However, we provide a sufficient condition next for which a smooth MLL parameterization or a curved exponential family can be obtained under collapsibility conditions.

\begin{theorem}\label{th5.2}
Let $\{A,B\}$ be a partition of $M$. Then there exists a smooth MLL parameterization of $\mathcal{F}$ on $\mathfrak{X}_{M}$ under collapsibility conditions with respect to $L\subseteq M$ if $X_{A}\ci X_{v}|X_{B\backslash\{v\}}$ for some $v\in L$. Also the parameterization is given by $\{\tilde{\lambda}_{L}^{M}\mid L\in\mathbb{P}(M)\backslash\mathcal{D}\}$ where $\mathbb{P}(.)$ denotes the power set and $\mathcal{D}$ is the collection of all sets of the form $A'vB'$ with $\emptyset\neq A'\subseteq A$ and $\emptyset\neq B'\subseteq B\backslash \{v\}$. 
\end{theorem}

\begin{remark}\label{rem5.1}
For a complete but non-hierarchical $\mathcal{P}$, define $\mathcal{P}_{-L} = \{(L',M\backslash L)\mid (L',M)\in\mathcal{P},L\not\subseteq L'\}$. Then by Proposition 3.5 of Evans (2015), $\tilde{\lambda}_{\mathcal{P}}$ is a smooth parameterization of $\mathfrak{X}_{M}$ if and only if $\tilde{\lambda}_{\mathcal{P}_{-L}}$ is a smooth parameterization of $\mathfrak{X}_{M\backslash L}$. 
\end{remark}

Using Theorem \ref{th5.2}, we next provide a sufficient condition for collapsibility of a multidimensional table with respect to MLL parameters as shown below. 

\begin{theorem}\label{th5.3}
Let $A$, $B$ and $C$ form a partition of $M$. Then for $R\in\{A,B,C\}$, the table $\mathfrak{X}_{M}$ is collapsible over $\mathfrak{X}_{R}$ into $\mathfrak{X}_{M\backslash R}$ with respect to $\{\lambda_{L}^{M}|L\subseteq M\backslash R\}$ if $X_{R}\ci X_{v}|X_{(M\backslash R)\backslash\{v\}}$ for some $v\in M\backslash R$.
\end{theorem}


\section{Conclusions}
In this paper, our main aim has been to investigate collapsibility for categorical data in a multidimensional contingency table. For this purpose, we consider a large class of models called marginal models introduced by Bergsma and Rudas (2002) for studying strictly positive discrete distributions on such tables. The MLL parameters include the ordinary log-linear and multivariate logistic parameters as special cases. Moreover, the marginal models also generalize several other models studied in the literature. For the multidimensional table, it is assumed that each categorical variable has an arbitrary number of levels.
 
First, we obtain some distinctive properties of MLL parameters using simple expressions for such parameters. Then collapsibility and strict collapsibility of these parameters are defined in a general sense by considering two arbitrary margins of a table. We derive several necessary and sufficient conditions for collapsibility and strict collapsibility, which are easily verifiable from a table since they involve only simple functions of the cell probabilities. The MLE's of these probabilities either have closed-form expressions under some models, or can be computed using iterative procedures.  We also provide various results on collapsibility and strict collapsibility with respect to an arbitrary set of MLL parameters containing some common effects. Such results are useful for studying associations among various categorical variables in a table. Further, we explore the relationship of strict collapsibility with various forms of independence of the variables. We establish necessary and sufficient conditions for each case. All the above results are illustrated by analyzing some real-life datasets. Finally, we provide a general result on the connection between parameters having a common effect but defined within different margins. This result is then used to obtain a smooth MLL parameterization or a curved exponential family under collapsibility conditions. A sufficient condition for collapsibility in a multidimensional table is also provided using the result.

\section*{Appendix}

\subsection*{Proof of Lemma \ref{lem1.1}:} 
\begin{proof}
Let $L_{v}=L\backslash\{v\}$ and $L'_{v}\subseteq L_{v}$ for any $v\in L$. To prove the result, we need the following facts. \\
1. For $L'\subseteq L$ and $v\not\in L'$, we have $L'=L'_{v}$, $x_{v}\in\mathfrak{X}_{L'^{c}}=\mathfrak{X}_{{L'_{v}}^{c}}\subset\mathfrak{X}_{M\backslash L'_{v}} $ and
\begin{align}\label{eq1.9}
\frac{1}{|\mathfrak{X}_{v}|}\sum_{x_{v}}\nu_{L'}^{M}(x_{L'})&=\frac{1}{|\mathfrak{X}_{v}|}\sum_{x_{v}}\left[\frac{1}{|\mathfrak{X}_{M\backslash L'}|}\sum_{j_{M}\in\mathfrak{X}_{M}:j_{L'}=x_{L'}}\log p_{M}(j_{M})\right]\quad(\textrm{from}~(\ref{eq1.3})) \nonumber \\
&=\frac{1}{|\mathfrak{X}_{v}|}\times|\mathfrak{X}_{v}|\times\left[\frac{1}{|\mathfrak{X}_{M\backslash L'_{v}}|}\sum_{j_{M}\in\mathfrak{X}_{M}:j_{L'_{v}}=x_{L'_{v}}}\log p_{M}(j_{M})\right] (\because L'=L'_{v},~x_{v}\in\mathfrak{X}_{M\backslash L'_{v}})\nonumber \\
&=\nu_{L'_{v}}^{M}(x_{L'_{v}}).
\end{align} 
2. For $L'\subseteq L$ and $v\in L'$, we have $L'={L'_{v}}\cup\{v\}$, $x_{v}\in\mathfrak{X}_{L'}=\mathfrak{X}_{{L'_{v}}\cup\{v\}}$ and
\begin{align}\label{eq1.10}
\frac{1}{|\mathfrak{X}_{v}|}\sum_{x_{v}}\nu_{L'}^{M}(x_{L'})&=\frac{1}{|\mathfrak{X}_{v}|}\sum_{x_{v}}\left[\frac{1}{|\mathfrak{X}_{M\backslash L'}|}\sum_{j_{M}\in\mathfrak{X}_{M}:j_{L'}=x_{L'}}\log p_{M}(j_{M})\right]\quad(\textrm{from}~(\ref{eq1.3})) \nonumber \\
&=\frac{1}{|\mathfrak{X}_{v}|}\times\frac{1}{|\mathfrak{X}_{M\backslash L'}|}\sum_{x_{v}}\sum_{j_{M}\in\mathfrak{X}_{M}:j_{L'}=x_{L'}}\log p_{M}(j_{M}) \nonumber \\
&=\frac{1}{|\mathfrak{X}_{M\backslash L'_{v}}|}\sum_{j_{M}\in\mathfrak{X}_{M}:j_{L'_{v}}=x_{L'_{v}}}\log p_{M}(j_{M})\quad(\textrm{see explanation below}) \nonumber \\
&=\nu_{L'_{v}}^{M}(x_{L'_{v}}).
\end{align}
The second last line of (\ref{eq1.10}) follows from the fact that $M\backslash L'_{v}=M\cap {L'_{v}}^{c}=M\cap (L'\backslash\{v\})^{c}=M\cap(L'\cap\{v\}^{c})^{c}=M\cap(L'^{c}\cup\{v\})=(M\cap L'^{c})\cup\{v\}=(M\backslash L')\cup\{v\}$, which implies $|\mathfrak{X}_{M\backslash L'_{v}}|=|\mathfrak{X}_{(M\backslash L')\cup \{v\}}|=|\mathfrak{X}_{M\backslash L'}||\mathfrak{X}_{v}|$. Similarly, $\sum_{x_{v}}\sum_{j_{M}\in\mathfrak{X}_{M}:j_{L'}=x_{L'}}\log p_{M}(j_{M})=\sum_{j_{M}\in\mathfrak{X}_{(M\backslash L')\cup \{v\}}}\log p_{M}(j_{M})=\sum_{j_{M}\in\mathfrak{X}_{M\backslash L'_{v}}}\log p_{M}(j_{M})=\sum_{j_{M}\in\mathfrak{X}_{M}:j_{L'_{v}}=x_{L'_{v}}}\log p_{M}(j_{M})$. Also, note that $\{L'|L'\subseteq L;v\in L'\}=\{L'_{v}\cup\{v\}|L'_{v}\subseteq L_{v}\}$ and $\{L'|L'\subseteq L;v\not\in L'\}=\{L'_{v}|L'_{v}\subseteq L_{v}\}$. Hence, using (\ref{eq1.9}) and (\ref{eq1.10}), we have from (\ref{eq1.4}), 
\begin{align*}
\sum_{x_{v}}\lambda_{L}^{M}(x_{L})&=\sum_{x_{v}}\sum_{L'\subseteq L}(-1)^{|L\backslash L'|}\nu_{L'}^{M}(x_{L'}) \\
&=\sum_{x_{v}}\left[\sum_{L'\subseteq L;v\in L'}(-1)^{|L\backslash L'|}\nu_{L'}^{M}(x_{L'})+\sum_{L'\subseteq L;v\not\in L'}(-1)^{|L\backslash L'|}\nu_{L'}^{M}(x_{L'})\right] \\
&=|\mathfrak{X}_{v}|\left[\sum_{\substack{L'\subseteq L\\v\in L'}}(-1)^{|L\backslash L'|}\left(\frac{1}{|\mathfrak{X}_{v}|}\sum_{x_{v}}\nu_{L'}^{M}(x_{L'})\right)+\sum_{\substack{L'\subseteq L\\v\not\in L'}}(-1)^{|L\backslash L'|}\left(\frac{1}{|\mathfrak{X}_{v}|}\sum_{x_{v}}\nu_{L'}^{M}(x_{L'})\right)\right] \\
&=|\mathfrak{X}_{v}|\left[\sum_{L'_{v}\subseteq L_{v}}(-1)^{|L_{v}\backslash L'_{v}|}\nu_{L'_{v}}^{M}(x_{L'_{v}}) + \sum_{L'_{v}\subseteq L_{v}}(-1)^{|(L_{v}\cup\{v\})\backslash L'_{v}|}\nu_{L'_{v}}^{M}(x_{L'_{v}})\right] \\
&=|\mathfrak{X}_{v}|\left[\sum_{L'_{v}\subseteq L_{v}}\{(-1)^{|L_{v}\backslash L'_{v}|}(1-1)\}\nu_{L'_{v}}^{M}(x_{L'_{v}})\right] \\
&=0, 
\end{align*}
which completes the proof. 
\end{proof}

\subsection*{Proof of Lemma \ref{lem1.1a}:} 
\begin{proof}
The reverse implication is obvious. For the forward implication, we assume
\begin{equation}\label{eq1.10b}
\sum_{L'\subseteq L}\lambda_{L'}^{M}(x_{L'})=\sum_{L'\subseteq L}\lambda_{L'}^{N}(x_{L'}).
\end{equation}
First fix $L'\subset L$. Then summing both sides of (\ref{eq1.10b}) over $x_{L''}\subset x_{L}$ with $L''\neq L'$, we get from Lemma \ref{lem1.1}
\begin{equation}\label{eq1.10c}
\lambda_{L'}^{M}(x_{L'})=\lambda_{L'}^{N}(x_{L'}).
\end{equation}		
Substituting (\ref{eq1.10c}) in (\ref{eq1.10b}) and repeating the steps for other $L'\subset L$, we obtain
\begin{equation}\label{eq1.10d}
\lambda_{L'}^{M}(x_{L'})=\lambda_{L'}^{N}(x_{L'})\quad\forall~\emptyset\neq L'\subset L.
\end{equation}
Summing both sides of (\ref{eq1.10d}) over $L'\subset L$ and then using (\ref{eq1.10b}), we obtain $\lambda_{L}^{M}(x_{L})=\lambda_{L}^{N}(x_{L})$, which along with (\ref{eq1.10d}) proves the result.
\end{proof}

\subsection*{Proof of Lemma \ref{lem1.2}:} 
\begin{proof}
Let (\ref{eq1.11}) hold, where $\lambda_{L}^{M}$ satisfies Lemma \ref{lem1.1}. Then for $A\subseteq M$,
\begin{align*}
\nu_{A}^{M}(x_{A})&=\frac{1}{|\mathfrak{X}_{M\backslash A}|}\sum_{j_{M}\in\mathfrak{X}_{M}:j_{A}=x_{A}}\log p_{M}(j_{M})\\
&=\frac{1}{|\mathfrak{X}_{M\backslash A}|}\sum_{j_{M}\in\mathfrak{X}_{M}:j_{A}=x_{A}}\sum_{L\subseteq M}\lambda_{L}^{M}(j_{L}) \\
&=\frac{1}{|\mathfrak{X}_{M\backslash A}|}\left[\sum_{j_{M}\in\mathfrak{X}_{M}:j_{A}=x_{A}}\left\{\sum_{L\subseteq A}\lambda_{L}^{M}(j_{L})+\sum_{L\not\subset A}\lambda_{L}^{M}(j_{L})\right\}\right] \\
&=\frac{|\mathfrak{X}_{M\backslash A}|}{|\mathfrak{X}_{M\backslash A}|}\sum_{L\subseteq A}\lambda_{L}^{M}(x_{L})+\frac{1}{|\mathfrak{X}_{M\backslash A}|}\sum_{L\not\subset A}\left\{\sum_{\substack{j_{M}\in\mathfrak{X}_{M}\\j_{A}=x_{A}}}\lambda_{L}^{M}(j_{L})\right\}~~(\because L\subseteq A\Rightarrow L\cap(M\backslash A)=\emptyset) \\
&=\sum_{L\subseteq A}\lambda_{L}^{M}(x_{L})
\end{align*}
since $L\not\subset A\Rightarrow L\cap(M\backslash A)\neq\emptyset\Rightarrow \sum_{j_{M}\in\mathfrak{X}_{M}:j_{A}=x_{A}}\lambda_{L}^{M}(j_{L})=0$ by Lemma \ref{lem1.1}. 
For the sufficiency part, observe that by substituting $A=M$ in (\ref{eq1.12}), we have $\nu_{M}^{M}=\log p_{M}(x_{M})$ (LHS of (\ref{eq1.11})) using (\ref{eq1.3}). Also, the RHS of (\ref{eq1.11}) and (\ref{eq1.12}) become identical. 
\end{proof}

\subsection*{Proof of Theorem \ref{th2.1}:}
\begin{proof}
	From (\ref{eq2.3}), we have
	\begin{equation}\label{eq2.5}
	\log p_{M}(x_{M})=\nu_{M}^{N}(x_{M}) + d^{(M)}(x_{M}).
	\end{equation}	
Also for any $Z\subseteq M$, 
\begin{align*}
\nu_{Z}^{M}(x_{Z}) &= \frac{1}{|\mathfrak{X}_{M\backslash Z}|}\sum_{j_{M}\in\mathfrak{X}_{M}:j_{Z}=x_{Z}}\log p_{M}{(j_{M}) } \nonumber \\
&=\frac{1}{|\mathfrak{X}_{M\backslash Z}|}\sum_{j_{M}\in\mathfrak{X}_{M}:j_{Z}=x_{Z}}\left[\nu_{M}^{N}(j_{M}) + d^{(M)}(j_{M})\right]\quad(\textrm{from}~(\ref{eq2.5})) \nonumber \\
&=\frac{1}{|\mathfrak{X}_{M\backslash Z}||\mathfrak{X}_{N\backslash M}|}\sum_{\substack{j_{M}\in\mathfrak{X}_{M}\\j_{Z}=x_{Z}}}\sum_{\substack{i_{N}\in\mathfrak{X}_{N}\\i_{M}=x_{M}}}\log p_{N}(i_{N}) + \frac{1}{|\mathfrak{X}_{M\backslash Z}|}\sum_{\substack{j_{M}\in\mathfrak{X}_{M}\\j_{Z}=x_{Z}}}d^{(M)}(j_{M}).
\end{align*}
Since $(M\backslash Z)\cap (N\backslash M)=\emptyset$ and $(M\backslash Z)\cup (N\backslash M)=N\backslash Z$, we obtain
\begin{align}\label{eq2.6}
\nu_{Z}^{M}(x_{Z})&=\frac{1}{|\mathfrak{X}_{N\backslash Z}|}\sum_{\substack{j_{N}\in\mathfrak{X}_{N}\\j_{Z}=x_{Z}}}\log p_{N}(j_{N}) + {\tilde{d}_{Z}}^{(M)}(x_{Z} )\quad(\textrm{from}~(\ref{eq2.4})) \nonumber \\
&= \nu_{Z}^{N}(x_{Z})  + {\tilde{d}_{Z}}^{(M)}(x_{Z} ). 		 
\end{align}
	The rest of the proof uses Lemma \ref{lem1.2} and M\"{o}bius inversion, and is similar to the proof of Theorem 3.1 in Vellaisamy and Vijay (2007).
\end{proof}

\subsection*{Proof of Theorem \ref{th2.2}:} 
\begin{proof}
Using (\ref{eq2.6}), note that the condition (\ref{eq2.9}) is equivalent to
	\begin{equation}\label{eq2.10}
	\sum_{R\subseteq S}(-1)^{|S\backslash R|}\left[\nu_{L_{R}}^{M}(x_{L_{R}})-\nu_{L_{R}}^{N}(x_{L_{R}})\right]=0.
	\end{equation} 
	Using Lemma \ref{lem1.2}, (\ref{eq2.10}) reduces to
	\begin{equation}\label{eq2.11}
	\sum_{R\subseteq S}(-1)^{|S\backslash R|}\left[\sum_{Z\subseteq L_{R}}\lambda_{Z}^{M}(x_{Z})-\sum_{Z\subseteq L_{R}}\lambda_{Z}^{N}(x_{Z})\right]=0
	\end{equation}
	or equivalently
	\begin{equation}\label{eq2.12}
	\sum_{Z\subseteq L_{S}}\lambda_{Z\cup S}^{M}(x_{Z},x_{S})=\sum_{Z\subseteq L_{S}}\lambda_{Z\cup S}^{N}(x_{Z},x_{S}).
	\end{equation}
	The subsequent arguments required to prove the result use Lemma \ref{lem1.1} and follow from those in the proofs of Theorems 3.2 and 3.3 in Vellaisamy and Vijay (2007).
\end{proof}

\subsection*{Proof of Lemma \ref{lem3.1}:} 
\begin{proof}
The arguments are similar to those used for the proof of Lemma 4.1 in Vellaisamy and Vijay (2007).
\end{proof}

\subsection*{Proof of Theorem \ref{th3.1}:} 
\begin{proof}
The result follows from Lemma \ref{lem3.1} and similar arguments used in the proof of Theorem 4.1 in Vellaisamy and Vijay (2007).
\end{proof}

\subsection*{Proof of Theorem \ref{th3.2}:}
\begin{proof}
The result can be proved using Theorems \ref{th2.2} and \ref{th3.1}, condition (ii) of Definition \ref{def3.1} and similar arguments as in the proof of Theorem 4.2 in Vellaisamy and Vijay (2007). 
\end{proof}

\subsection*{Proof of Theorem \ref{th4.1}:} 
\begin{proof}
(a) Sufficiency: \\	
From Lemma \ref{lem1.11}, for the $|M|$-dimensional table, we have
\begin{align}\label{eq4.1}
\log p_{M}(x_{M}) &= \sum_{L\subseteq ABC}\lambda_{L}^{ABC}(x_{L}) \nonumber \\
&= \lambda_{\emptyset}^{ABC} + \sum_{A'\subseteq A}\lambda_{A'}^{ABC}(x_{A'}) + \sum_{B'\subseteq B}\lambda_{B'}^{ABC}(x_{B'}) + \sum_{C'\subseteq C}\lambda_{C'}^{ABC}(x_{C'}) \nonumber \\
& + \sum_{A'B'\subseteq AB}\lambda_{A'B'}^{ABC}(x_{A'B'}) + \sum_{A'C'\subseteq AC}\lambda_{A'C'}^{ABC}(x_{A'C'}) + \sum_{B'C'\subseteq BC}\lambda_{B'C'}^{ABC}(x_{B'C'}) \nonumber \\
& +\sum_{A'B'C'\subseteq ABC}\lambda_{A'B'C'}^{ABC}(x_{A'B'C'})
\end{align}	
for non-empty subsets $A'$, $B'$ and $C'$ of $A$, $B$ and $C$ respectively. Using Lemma \ref{lem1.3}, we have
\begin{equation}\label{eq4.2}
X_{A}\ci X_{B}|X_{C}\Leftrightarrow \lambda_{A'B'}^{ABC}=\lambda_{A'B'C'}^{ABC}=0\quad\forall~\emptyset\neq A'\subseteq A,~\emptyset\neq B'\subseteq B,~\emptyset\neq C'\subseteq C.
\end{equation}
From (\ref{eq4.1}) and (\ref{eq4.2}), we obtain under $X_{A}\ci X_{B}|X_{C}$ 
\begin{align}\label{eq4.3}
\log p_{M}(x_{M}) &=\lambda_{\emptyset}^{ABC} + \sum_{A'\subseteq A}\lambda_{A'}^{ABC}(x_{A'}) + \sum_{B'\subseteq B}\lambda_{B'}^{ABC}(x_{B'}) + \sum_{C'\subseteq C}\lambda_{C'}^{ABC}(x_{C'}) \nonumber \\
& + \sum_{A'C'\subseteq AC}\lambda_{A'C'}^{ABC}(x_{A'C'}) + \sum_{B'C'\subseteq BC}\lambda_{B'C'}^{ABC}(x_{B'C'}).
\end{align}
The logarithms of the marginal probabilities in (\ref{eq4.3}) are
\begin{equation}\label{eq4.4}
\log p_{AB}(x_{AB}) = \lambda_{\emptyset}^{ABC} + \sum_{A'\subseteq A}\lambda_{A'}^{ABC}(x_{A'}) + \sum_{B'\subseteq B}\lambda_{B'}^{ABC}(x_{B'}) + \lambda(x_{A},x_{B}), 
\end{equation}
\begin{equation}\label{eq4.5}
\log p_{AC}(x_{AC}) = \lambda_{\emptyset}^{ABC} + \sum_{A'\subseteq A}\lambda_{A'}^{ABC}(x_{A'}) + \sum_{C'\subseteq C}\lambda_{C'}^{ABC}(x_{C'}) + \sum_{A'C'\subseteq AC}\lambda_{A'C'}^{ABC}(x_{A'C'})  + \lambda(x_{C}), 
\end{equation}
\begin{equation}\label{eq4.6}
\log p_{BC}(x_{BC}) = \lambda_{\emptyset}^{ABC} + \sum_{B'\subseteq B}\lambda_{B'}^{ABC}(x_{B'}) + \sum_{C'\subseteq C}\lambda_{C'}^{ABC}(x_{C'}) + \sum_{B'C'\subseteq BC}\lambda_{B'C'}^{ABC}(x_{B'C'})  + \lambda'(x_{C}), 
\end{equation}
where
\begin{align*}
\lambda(x_{A},x_{B}) &= \log\left(\sum_{x_{C}}\exp\left\{\sum_{C'\subseteq C}\lambda_{C'}^{ABC}(x_{C'}) + \sum_{A'C'\subseteq AC}\lambda_{A'C'}^{ABC}(x_{A'C'}) + \sum_{B'C'\subseteq BC}\lambda_{B'C'}^{ABC}(x_{B'C'})\right\}\right), \\
\lambda(x_{C}) &= \log\left(\sum_{x_{B}}\exp\left\{\sum_{B'\subseteq B}\lambda_{B'}^{ABC}(x_{B'})+\sum_{B'C'\subseteq BC}\lambda_{B'C'}^{ABC}(x_{B'C'})\right\}\right), \\
\lambda'(x_{C}) &= \log\left(\sum_{x_{A}}\exp\left\{\sum_{A'\subseteq A}\lambda_{A'}^{ABC}(x_{A'})+\sum_{A'C'\subseteq AC}\lambda_{A'C'}^{ABC}(x_{A'C'})\right\}\right).
\end{align*}
If we collapse the $|M|$-dimensional table over $\mathfrak{X}_{A}$, we get
\begin{equation}\label{eq4.7}
\log p_{BC}(x_{BC}) = \lambda_{\emptyset}^{BC} + \sum_{B'\subseteq B}\lambda_{B'}^{BC}(x_{B'}) + \sum_{C'\subseteq C}\lambda_{C'}^{BC}(x_{C'}) + \sum_{B'C'\subseteq BC}\lambda_{B'C'}^{BC}(x_{B'C'}).
\end{equation}
We now compare (\ref{eq4.6}) and (\ref{eq4.7}). Summing RHS of both over $x_{B}$ and $x_{C}$ gives
\begin{equation}\label{eq4.8}
\lambda_{\emptyset}^{BC} = \lambda_{\emptyset}^{ABC} + \frac{\sum_{x_{C}}\lambda'(x_{C})}{|\mathfrak{X}_{C}|}.
\end{equation}
Summing RHS of (\ref{eq4.6}) and (\ref{eq4.7}) over $x_{C}$ only, we have
\begin{align}\label{eq4.9}
|\mathfrak{X}_{C}|\left(\lambda_{\emptyset}^{BC} + \sum_{B'\subseteq B}\lambda_{B'}^{BC}(x_{B'})\right) &= |\mathfrak{X}_{C}|\left(\lambda_{\emptyset}^{ABC} + \sum_{B'\subseteq B}\lambda_{B'}^{ABC}(x_{B'})\right) + \sum_{x_{C}}\lambda'(x_{C}) \nonumber \\
\Rightarrow \sum_{B'\subseteq B}\lambda_{B'}^{BC}(x_{B'}) &= \sum_{B'\subseteq B}\lambda_{B'}^{ABC}(x_{B'})\quad(\textrm{using}~(\ref{eq4.8})).
\end{align}
Using Lemma \ref{lem1.1a}, we have from (\ref{eq4.9})
\begin{equation}\label{eq4.11}
\lambda_{B'}^{BC}(x_{B'})=\lambda_{B'}^{ABC}(x_{B'})\quad\forall~B'\subseteq B.
\end{equation}
Now summing RHS of (\ref{eq4.6}) and (\ref{eq4.7}) over $x_{B}$ only gives
\begin{align}\label{eq4.12}
|\mathfrak{X}_{B}|\left(\lambda_{\emptyset}^{BC} + \sum_{C'\subseteq C}\lambda_{C'}^{BC}(x_{C'})\right) &= |\mathfrak{X}_{B}|\left(\lambda_{\emptyset}^{ABC} + \sum_{C'\subseteq C}\lambda_{C'}^{ABC}(x_{C'})+\lambda'(x_{C})\right) \nonumber \\
\Rightarrow\sum_{C'\subseteq C}\lambda_{C'}^{BC}(x_{C'})&=\sum_{C'\subseteq C}\lambda_{C'}^{ABC}(x_{C'}) + \lambda'(x_{C}) - \frac{\sum_{x_{C}}\lambda'(x_{C})}{|\mathfrak{X}_{C}|}\quad(\textrm{using}~(\ref{eq4.8})). 
\end{align}
Summing both sides of (\ref{eq4.8}), (\ref{eq4.9}) and (\ref{eq4.12}), we get
\begin{equation}\label{eq4.13}
\lambda_{\emptyset}^{BC}+\sum_{B'\subseteq B}\lambda_{B'}^{BC}(x_{B'})+\sum_{C'\subseteq C}\lambda_{C'}^{BC}(x_{C'})=\lambda_{\emptyset}^{BC}+\sum_{B'\subseteq B}\lambda_{B'}^{BC}(x_{B'})+\sum_{C'\subseteq C}\lambda_{C'}^{BC}(x_{C'}) + \lambda'(x_{C}).
\end{equation}
From (\ref{eq4.6}), (\ref{eq4.7}) and (\ref{eq4.13}), we have
\begin{equation}\label{eq4.14}
\sum_{B'C'\subseteq BC}\lambda_{B'C'}^{BC}(x_{B'C'})=\sum_{B'C'\subseteq BC}\lambda_{B'C'}^{ABC}(x_{B'C'}). 
\end{equation}
Using Lemma \ref{lem1.1a}, it can be shown from (\ref{eq4.14}) that
\begin{equation}\label{eq4.15}
\lambda_{B'C'}^{BC}(x_{B'C'})=\lambda_{B'C'}^{ABC}(x_{B'C'})\quad\forall~B'C'\subseteq BC.
\end{equation}
Analogous results can be obtained by collapsing $\mathfrak{X}_{M}$ over $\mathfrak{X}_{B}$ and then comparing $p_{AC}(x_{AC})$ in $\mathfrak{X}_{M}$ and $\mathfrak{X}_{AC}$. In this case, we get
\begin{equation}\label{eq4.16}
\lambda_{A'}^{AC}(x_{A'}) = \lambda_{A'}^{ABC}(x_{A'})~\forall~A'\subseteq A; ~\lambda_{A'C'}^{AC}(x_{A'C'}) = \lambda_{A'C'}^{ABC}(x_{A'C'})~\forall~A'C'\subseteq AC.  
\end{equation}
Hence, collapsibility over $\mathfrak{X}_{A}$ ($\mathfrak{X}_{B}$) follows from (\ref{eq4.11}) and (\ref{eq4.15}) ((\ref{eq4.16})). \\
Since $B'\subset A'B'\not\subseteq BC$, $B'C'\subset A'B'C'\not\subseteq BC$ and $\lambda_{A'B'}^{ABC}=\lambda_{A'B'C'}^{ABC}=0$ (see (\ref{eq4.2})), strict collapsibility over $\mathfrak{X}_{A}$ with respect to $\lambda_{B'}^{ABC}$ and $\lambda_{B'C'}^{ABC}$ follows from Definition \ref{def3.1}. \\
Also, since $A'\subset A'B'\not\subseteq AC$, $A'C'\subset A'B'C'\not\subseteq AC$ and $\lambda_{A'B'}^{ABC}=\lambda_{A'B'C'}^{ABC}=0$ (see (\ref{eq4.2})), strict collapsibility over $\mathfrak{X}_{B}$ with respect to $\lambda_{A'}^{ABC}$ and $\lambda_{A'C'}^{ABC}$ follows. 
\vone\noindent
(b) Necessity: \\
Since $\mathfrak{X}_{M}$ is strictly collapsible over $\mathfrak{X}_{A}$ ($\mathfrak{X}_{B}$) with respect to $\lambda_{B'}^{ABC}$ and $\lambda_{B'C'}^{ABC}$ ($\lambda_{A'}^{ABC}$ and $\lambda_{A'C'}^{ABC}$), we have
\begin{enumerate}
	\item[1.]  $\lambda_{B'}^{BC}=\lambda_{B'}^{ABC}$ and $\lambda_{B'C'}^{BC}=\lambda_{B'C'}^{ABC}$ ($\lambda_{A'}^{AC}=\lambda_{A'}^{ABC}$ and $\lambda_{A'C'}^{AC}=\lambda_{A'C'}^{ABC}$), 
	\item[2.] $\lambda_{A'B'}^{ABC} = \lambda_{A'B'C'}^{ABC} = 0$. 
\end{enumerate}
From Point 2 above and using (\ref{eq4.2}), $X_{A}\ci X_{B}|X_{C}$. 
\end{proof}

\subsection*{Proof of Theorem \ref{th4.2}:} 
\begin{proof}
(a) Sufficiency: \\
The MLL parameters $\lambda_{L}^{M}$ for $\mathfrak{X}_{M}$ satisfy (\ref{eq4.1}). Using Lemma \ref{lem1.3}, we have
\begin{equation}\label{eq4.17}
X_{A}\ci (X_{B},X_{C})\Leftrightarrow\lambda_{A'B'}^{ABC}=\lambda_{A'C'}^{ABC}=\lambda_{A'B'C'}^{ABC}=0\quad\forall~\emptyset\neq A'\subseteq A,~\emptyset\neq B'\subseteq B,~\emptyset\neq C'\subseteq C.
\end{equation}	
Hence, from (\ref{eq4.1}) and (\ref{eq4.17}), we obtain under $X_{A}\ci (X_{B},X_{C})$ 
\begin{align}\label{eq4.18}
\log p_{M}(x_{M}) &=\lambda_{\emptyset}^{ABC} + \sum_{A'\subseteq A}\lambda_{A'}^{ABC}(x_{A'}) + \sum_{B'\subseteq B}\lambda_{B'}^{ABC}(x_{B'}) + \sum_{C'\subseteq C}\lambda_{C'}^{ABC}(x_{C'}) + \sum_{B'C'\subseteq BC}\lambda_{B'C'}^{ABC}(x_{B'C'}). 
\end{align} 
The logarithms of the two-dimensional marginal probabilities in (\ref{eq4.18}) are
\begin{equation}\label{eq4.19}
\log p_{AB}(x_{AB}) = \lambda_{\emptyset}^{ABC} + \sum_{A'\subseteq A}\lambda_{A'}^{ABC}(x_{A'}) + \sum_{B'\subseteq B}\lambda_{B'}^{ABC}(x_{B'}) + \lambda(x_{B}), 
\end{equation}
\begin{equation}\label{eq4.20}
\log p_{AC}(x_{AC}) = \lambda_{\emptyset}^{ABC} + \sum_{A'\subseteq A}\lambda_{A'}^{ABC}(x_{A'}) + \sum_{C'\subseteq C}\lambda_{C'}^{ABC}(x_{C'}) + \lambda(x_{C}), 
\end{equation}
\begin{equation}\label{eq4.21}
\log p_{BC}(x_{BC}) = \lambda_{\emptyset}^{ABC} + \sum_{B'\subseteq B}\lambda_{B'}^{ABC}(x_{B'}) + \sum_{C'\subseteq C}\lambda_{C'}^{ABC}(x_{C'}) + \sum_{B'C'\subseteq BC}\lambda_{B'C'}^{ABC}(x_{B'C'}) + \lambda(x_{A}),
\end{equation} 
where
\begin{align*}
\lambda(x_{B}) &= \log\left(\sum_{x_{C}}\exp\left\{\sum_{C'\subseteq C}\lambda_{C'}^{ABC}(x_{C'}) + \sum_{B'C'\subseteq BC}\lambda_{B'C'}^{ABC}(x_{B'C'})\right\}\right), \\
\lambda(x_{C}) &= \log\left(\sum_{x_{B}}\exp\left\{\sum_{B'\subseteq B}\lambda_{B'}^{ABC}(x_{B'})+\sum_{B'C'\subseteq BC}\lambda_{B'C'}^{ABC}(x_{B'C'})\right\}\right), \\
\lambda(x_{A}) &= \log\left(\sum_{x_{A}}\exp\left\{\sum_{A'\subseteq A}\lambda_{A'}^{ABC}(x_{A'})\right\}\right).
\end{align*}
If we collapse $\mathfrak{X}_{M}$ over $\mathfrak{X}_{A}$, we get
\begin{equation}\label{eq4.22}
\log p_{BC}(x_{BC}) = \lambda_{\emptyset}^{BC} + \sum_{B'\subseteq B}\lambda_{B'}^{BC}(x_{B'}) + \sum_{C'\subseteq C}\lambda_{C'}^{BC}(x_{C'}) + \sum_{B'C'\subseteq BC}\lambda_{B'C'}^{BC}(x_{B'C'}).
\end{equation}
We now compare (\ref{eq4.21}) and (\ref{eq4.22}). Summing RHS of both over $x_{B}$ and $x_{C}$ gives
\begin{equation}\label{eq4.23}
\lambda_{\emptyset}^{BC} = \lambda_{\emptyset}^{ABC} + \lambda(x_{A}).
\end{equation}
Summing RHS of (\ref{eq4.21}) and (\ref{eq4.22}) over $x_{C}$ only, we have
\begin{align}\label{eq4.24}
|\mathfrak{X}_{C}|\left(\lambda_{\emptyset}^{BC} + \sum_{B'\subseteq B}\lambda_{B'}^{BC}(x_{B'})\right) &= |\mathfrak{X}_{C}|\left(\lambda_{\emptyset}^{ABC} + \sum_{B'\subseteq B}\lambda_{B'}^{ABC}(x_{B'}) + \lambda(x_{A})\right) \nonumber \\
\Rightarrow \sum_{B'\subseteq B}\lambda_{B'}^{BC}(x_{B'}) &= \sum_{B'\subseteq B}\lambda_{B'}^{ABC}(x_{B'})\quad(\textrm{using}~(\ref{eq4.23})).
\end{align}
Using Lemma \ref{lem1.1a}, we have from (\ref{eq4.24})
\begin{equation}\label{eq4.26}
\lambda_{B'}^{BC}(x_{B'})=\lambda_{B'}^{ABC}(x_{B'})\quad\forall~B'\subseteq B.
\end{equation}
Now summing RHS of (\ref{eq4.21}) and (\ref{eq4.22}) over $x_{B}$ only gives
\begin{align}\label{eq4.27}
|\mathfrak{X}_{B}|\left(\lambda_{\emptyset}^{BC} + \sum_{C'\subseteq C}\lambda_{C'}^{BC}(x_{C'})\right) &= |\mathfrak{X}_{B}|\left(\lambda_{\emptyset}^{ABC} + \sum_{C'\subseteq C}\lambda_{C'}^{ABC}(x_{C'})+\lambda(x_{A})\right) \nonumber \\
\Rightarrow \sum_{C'\subseteq C}\lambda_{C'}^{BC}(x_{C'}) &= \sum_{C'\subseteq C}\lambda_{C'}^{ABC}(x_{C'})\quad(\textrm{using}~(\ref{eq4.23})).
\end{align}
Using Lemma \ref{lem1.1a}, it can be shown from (\ref{eq4.27}) that
\begin{equation}\label{eq4.28}
\lambda_{C'}^{BC}(x_{C'})=\lambda_{C'}^{ABC}(x_{C'})\quad\forall~C'\subseteq C.
\end{equation}
From (\ref{eq4.21})-(\ref{eq4.24}) and (\ref{eq4.27}), we have
\begin{equation}\label{eq4.28a}
\sum_{B'C'\subseteq BC}\lambda_{B'C'}^{BC}(x_{B'C'}) = \sum_{B'C'\subseteq BC}\lambda_{B'C'}^{ABC}(x_{B'C'}),
\end{equation}
which implies from Lemma \ref{lem1.1a}
\begin{equation}\label{eq4.28b}
\lambda_{B'C'}^{BC}(x_{B'C'}) = \lambda_{B'C'}^{ABC}(x_{B'C'})\quad\forall~B'C'\subseteq BC.
\end{equation}
Hence, collapsibility follows from (\ref{eq4.26}), (\ref{eq4.28}) and (\ref{eq4.28b}). \\
Since $B'\subset A'B'\not\subseteq BC$, $C'\subset A'C'\not\subseteq BC$ and $B',C',B'C'\subset A'B'C'\not\subseteq BC$ with $\lambda_{A'B'}^{ABC}=\lambda_{A'C'}^{ABC}=\lambda_{A'B'C'}^{ABC}=0$ (see (\ref{eq4.17})), strict collapsibility follows from Definition \ref{def3.1}. \\
\vone\noindent
(b) Necessity: \\
Strict collapsibility over $\mathfrak{X}_{A}$ with respect to $\lambda_{B'}^{ABC}$, $\lambda_{C'}^{ABC}$ and $\lambda_{B'C'}^{ABC}$ implies
\begin{enumerate}
	\item[1a.]  $\lambda_{B'}^{BC}=\lambda_{B'}^{ABC}$,  $\lambda_{C'}^{BC}=\lambda_{C'}^{ABC}$ and $\lambda_{B'C'}^{BC}=\lambda_{B'C'}^{ABC}$, 
	\item[1b.] $\lambda_{A'B'}^{ABC} = \lambda_{A'C'}^{ABC} = \lambda_{A'B'C'}^{ABC} = 0$. 
\end{enumerate}
From 1b above, we have $X_{A}\ci (X_{B},X_{C})$ using (\ref{eq4.17}). Hence, the result follows. 
\end{proof}

\subsection*{Proof of Theorem \ref{th4.3}:} 
\begin{proof}
Without loss of generality, we consider strict collapsibility with respect to MLL parameters in Parts 1 and 2 of Theorem \ref{th4.3}. The proof for Parts 1 and 3 or Parts 2 and 3 follows similarly. \\
(a) Sufficiency: \\
The MLL parameters corresponding to $\mathfrak{X}_{M}$ satisfy (\ref{eq4.1}). Using Lemma \ref{lem1.3}, we have
\begin{equation}\label{eq4.30}
X_{A}\ci X_{B}\ci X_{C}\Leftrightarrow\lambda_{A'B'}^{ABC}=\lambda_{A'C'}^{ABC}=\lambda_{B'C'}^{ABC}=\lambda_{A'B'C'}^{ABC}=0~\forall~\emptyset\neq A'\subseteq A,\emptyset\neq B'\subseteq B,\emptyset\neq C'\subseteq C.
\end{equation}	
Hence, from (\ref{eq4.1}) and (\ref{eq4.30}), we obtain under $X_{A}\ci X_{B}\ci X_{C}$ 
\begin{align}\label{eq4.31}
\log p_{M}(x_{M}) &=\lambda_{\emptyset}^{ABC} + \sum_{A'\subseteq A}\lambda_{A'}^{ABC}(x_{A'}) + \sum_{B'\subseteq B}\lambda_{B'}^{ABC}(x_{B'}) + \sum_{C'\subseteq C}\lambda_{C'}^{ABC}(x_{C'}). 
\end{align} 
The logarithms of the two-dimensional marginal probabilities in (\ref{eq4.31}) are
\begin{equation}\label{eq4.32}
\log p_{AB}(x_{AB}) = \lambda_{\emptyset}^{ABC} + \sum_{A'\subseteq A}\lambda_{A'}^{ABC}(x_{A'}) + \sum_{B'\subseteq B}\lambda_{B'}^{ABC}(x_{B'}) + \lambda(x_{C}), 
\end{equation}
\begin{equation}\label{eq4.33}
\log p_{AC}(x_{AC}) = \lambda_{\emptyset}^{ABC} + \sum_{A'\subseteq A}\lambda_{A'}^{ABC}(x_{A'}) + \sum_{C'\subseteq C}\lambda_{C'}^{ABC}(x_{C'}) + \lambda(x_{B}), 
\end{equation}
\begin{equation}\label{eq4.34}
\log p_{BC}(x_{BC}) = \lambda_{\emptyset}^{ABC} + \sum_{B'\subseteq B}\lambda_{B'}^{ABC}(x_{B'}) + \sum_{C'\subseteq C}\lambda_{C'}^{ABC}(x_{C'})  + \lambda(x_{A}),
\end{equation} 
where
\begin{align*}
\lambda(x_{C}) &= \log\left(\sum_{x_{C}}\exp\left\{\sum_{C'\subseteq C}\lambda_{C'}^{ABC}(x_{C'}) \right\}\right), \\
\lambda(x_{B}) &= \log\left(\sum_{x_{B}}\exp\left\{\sum_{B'\subseteq B}\lambda_{B'}^{ABC}(x_{B'})\right\}\right), \\
\lambda(x_{A}) &= \log\left(\sum_{x_{A}}\exp\left\{\sum_{A'\subseteq A}\lambda_{A'}^{ABC}(x_{A'})\right\}\right).
\end{align*}
Since $X_{A}\ci X_{B}\ci X_{C}\Rightarrow X_{B}\ci X_{C}\Leftrightarrow \lambda_{B'C'}^{BC}=0$, if we collapse $\mathfrak{X}_{M}$ over $\mathfrak{X}_{A}$, we get
\begin{equation}\label{eq4.35}
\log p_{BC}(x_{BC}) = \lambda_{\emptyset}^{BC} + \sum_{B'\subseteq B}\lambda_{B'}^{BC}(x_{B'}) + \sum_{C'\subseteq C}\lambda_{C'}^{BC}(x_{C'}).
\end{equation}
We now compare (\ref{eq4.34}) and (\ref{eq4.35}). Summing RHS of both over $x_{B}$ and $x_{C}$ gives
\begin{equation}\label{eq4.36}
\lambda_{\emptyset}^{BC} = \lambda_{\emptyset}^{ABC} + \lambda(x_{A}).
\end{equation}
Summing RHS of (\ref{eq4.34}) and (\ref{eq4.35}) over $x_{C}$ only, we have
\begin{align}\label{eq4.37}
|\mathfrak{X}_{C}|\left(\lambda_{\emptyset}^{BC} + \sum_{B'\subseteq B}\lambda_{B'}^{BC}(x_{B'})\right) &= |\mathfrak{X}_{C}|\left(\lambda_{\emptyset}^{ABC} + \sum_{B'\subseteq B}\lambda_{B'}^{ABC}(x_{B'}) + \lambda(x_{A})\right) \nonumber \\
\Rightarrow \sum_{B'\subseteq B}\lambda_{B'}^{BC}(x_{B'}) &= \sum_{B'\subseteq B}\lambda_{B'}^{ABC}(x_{B'})\quad(\textrm{using}~(\ref{eq4.36})).
\end{align}
Using Lemma \ref{lem1.1a}, we have from (\ref{eq4.37})
\begin{equation}\label{eq4.39}
\lambda_{B'}^{BC}(x_{B'})=\lambda_{B'}^{ABC}(x_{B'})\quad\forall~B'\subseteq B.
\end{equation}
Now summing RHS of (\ref{eq4.34}) and (\ref{eq4.35}) over $x_{B}$ only gives
\begin{align}\label{eq4.40}
|\mathfrak{X}_{B}|\left(\lambda_{\emptyset}^{BC} + \sum_{C'\subseteq C}\lambda_{C'}^{BC}(x_{C'})\right) &= |\mathfrak{X}_{B}|\left(\lambda_{\emptyset}^{ABC} + \sum_{C'\subseteq C}\lambda_{C'}^{ABC}(x_{C'})+\lambda(x_{A})\right) \nonumber \\
\Rightarrow \sum_{C'\subseteq C}\lambda_{C'}^{BC}(x_{C'}) &= \sum_{C'\subseteq C}\lambda_{C'}^{ABC}(x_{C'})\quad(\textrm{using}~(\ref{eq4.36})).
\end{align}
Using Lemma \ref{lem1.1a}, it can be shown from (\ref{eq4.40}) that
\begin{equation}\label{eq4.41}
\lambda_{C'}^{BC}(x_{C'})=\lambda_{C'}^{ABC}(x_{C'})\quad\forall~C'\subseteq C.
\end{equation}
Similarly, by collapsing $\mathfrak{X}_{M}$ over $\mathfrak{X}_{B}$, we get 
\begin{equation}\label{eq4.42}
\lambda_{A'}^{AC}(x_{A'})=\lambda_{A'}^{ABC}(x_{A'})~\textrm{and}~ \lambda_{C'}^{BC}(x_{C'})=\lambda_{C'}^{ABC}(x_{C'})\quad\forall~A',C'.
\end{equation}
Hence, collapsibility follows from (\ref{eq4.39}) and (\ref{eq4.41}) for Part 1, and from (\ref{eq4.42}) for Part 2. \\
Since $B'\subset A'B'\not\subseteq BC$, $C'\subset A'C'\not\subseteq BC$ and $B',C'\subset A'B'C'\not\subseteq BC$ with $\lambda_{A'B'}^{ABC}=\lambda_{A'C'}^{ABC}=\lambda_{A'B'C'}^{ABC}=0$ (see (\ref{eq4.30})), strict collapsibility follows for Part 1 from Definition \ref{def3.1}. \\
For Part 2, note that $A'\subset A'B'\not\subseteq AC$, $C'\subset B'C'\not\subseteq AC$ and $A',C'\subset A'B'C'\not\subseteq BC$ with  $\lambda_{A'B'}^{ABC}=\lambda_{B'C'}^{ABC}=\lambda_{A'B'C'}^{ABC}=0$ (see (\ref{eq4.30})) implying strict collapsibility. \\
\vone\noindent
(b) Necessity: \\
For Part 1, strict collapsibility over $\mathfrak{X}_{A}$ with respect to $\lambda_{B'}^{ABC}$ and $\lambda_{C'}^{ABC}$ implies
\begin{enumerate}
	\item[1a.]  $\lambda_{B'}^{BC}=\lambda_{B'}^{ABC}$ and $\lambda_{C'}^{BC}=\lambda_{C'}^{ABC}$, 
	\item[1b.] $\lambda_{A'B'}^{ABC} = \lambda_{A'C'}^{ABC} = \lambda_{A'B'C'}^{ABC} = 0$. 
\end{enumerate}
For Part 2, strict collapsibility over $\mathfrak{X}_{B}$ with respect to $\lambda_{A'}^{ABC}$ and $\lambda_{C'}^{ABC}$ implies
\begin{enumerate}
	\item[2a.]  $\lambda_{A'}^{AC}=\lambda_{A'}^{ABC}$ and $\lambda_{C'}^{AC}=\lambda_{C'}^{ABC}$, 
	\item[2b.] $\lambda_{A'B'}^{ABC} = \lambda_{B'C'}^{ABC} = \lambda_{A'B'C'}^{ABC} = 0$. 
\end{enumerate}
From 1b and 2b above, we have $\lambda_{A'B'}^{ABC} = \lambda_{A'C'}^{ABC} = \lambda_{B'C'}^{ABC} = \lambda_{A'B'C'}^{ABC} = 0\Leftrightarrow  X_{A}\ci X_{B}\ci X_{C}$ using (\ref{eq4.30}). Hence, the result follows. 
\end{proof} 

\subsection*{Proof of Theorem \ref{th5.1}:} 
\begin{proof}
Note that
\begin{align}\label{eq5.3}
\lambda_{L}^{AB}(x_{L}) &= \sum_{L'\subseteq L}(-1)^{|L\backslash L'|}\nu_{L'}^{AB}(x_{L'}) \nonumber \\
&= \sum_{L'\subseteq L}(-1)^{|L\backslash L'|}\frac{1}{|\mathfrak{X}_{AB\backslash L'}|}\sum_{\substack{j_{AB}\in\mathfrak{X}_{AB}\\j_{L'}=x_{L'}}}\log p_{AB}(j_{AB}) \nonumber \\
&= \sum_{L'\subseteq L}(-1)^{|L\backslash L'|}\frac{1}{|\mathfrak{X}_{A\backslash L'}||\mathfrak{X}_{B\backslash L'}|}\sum_{\substack{j_{AB}\in\mathfrak{X}_{AB}\\j_{L'}=x_{L'}}}[\log p_{B}(j_{B})+\log p_{A|B}(j_{A}|j_{B})] \nonumber \\
&= \sum_{L'\subseteq L}(-1)^{|L\backslash L'|}\frac{1}{|\mathfrak{X}_{A\backslash L'}||\mathfrak{X}_{B\backslash L'}|}\sum_{\substack{j_{A}\in\mathfrak{X}_{A}\\j_{L'}=x_{L'}}}\sum_{\substack{j_{B}\in\mathfrak{X}_{B}\\j_{L'}=x_{L'}}}\log p_{B}(j_{B}) \nonumber \\
& + \sum_{L'\subseteq L}(-1)^{|L\backslash L'|}\frac{1}{|\mathfrak{X}_{A\backslash L'}||\mathfrak{X}_{B\backslash L'}|}\sum_{\substack{j_{AB}\in\mathfrak{X}_{AB}\\j_{L'}=x_{L'}}}\log p_{A|B}(j_{A}|j_{B}) \nonumber \\
&= \sum_{L'\subseteq L}(-1)^{|L\backslash L'|}\frac{1}{|\mathfrak{X}_{A\backslash L'}||\mathfrak{X}_{B\backslash L'}|}\times|\mathfrak{X}_{A\backslash L'}|\sum_{\substack{j_{B}\in\mathfrak{X}_{B}\\j_{L'}=x_{L'}}}\log p_{B}(j_{B}) \nonumber \\
& + \sum_{L'\subseteq L}(-1)^{|L\backslash L'|}\frac{1}{|\mathfrak{X}_{AB\backslash L'}|}\sum_{\substack{j_{AB}\in\mathfrak{X}_{AB}\\j_{L'}=x_{L'}}}\log p_{A|B}(j_{A}|j_{B}) \nonumber \\
&= \lambda_{L}^{B}(x_{L}) + f(\lambda_{A|B})\quad\textrm{(say)}.
\end{align}	
Since the second term on the RHS of (\ref{eq5.3}) is a smooth function of the conditional probabilities $p_{A|B}(j_{A}|j_{B})$, it implies that $f$ is also a smooth function of the MLL parameters $\lambda_{A|B}$ defined in (\ref{eq5.1}). Suppose now $X_{A}\ci X_{v}|X_{B\backslash\{v\}}$ for some $v\in L$. Then
\begin{align}\label{eq5.4}
f(\lambda_{A|B}) &= \sum_{L'\subseteq L}(-1)^{|L\backslash L'|}\frac{1}{|\mathfrak{X}_{AB\backslash L'}|}\sum_{\substack{j_{AB}\in\mathfrak{X}_{AB}\\j_{L'}=x_{L'}}}\log p_{A|B}(j_{A}|j_{B}) \nonumber \\
&= \sum_{L'\subseteq L}(-1)^{|L\backslash L'|}\frac{1}{|\mathfrak{X}_{AB\backslash L}||\mathfrak{X}_{L\backslash L'}|}\sum_{\substack{j_{AB}\in\mathfrak{X}_{AB}\\j_{L'}=x_{L'}}}\log p_{A|B}(j_{A}|j_{v},j_{B\backslash\{v\}}) \nonumber \\
&=\frac{1}{|\mathfrak{X}_{AB\backslash L}|}\sum_{L'\subseteq L}(-1)^{|L\backslash L'|}\frac{1}{|\mathfrak{X}_{L\backslash L'}|}\sum_{\substack{j_{AB}\in\mathfrak{X}_{AB}\\j_{L'}=x_{L'}}}\log p_{A|B\backslash\{v\}}(j_{A}|j_{B\backslash\{v\}}) \nonumber \\ 
&=\frac{1}{|\mathfrak{X}_{AB\backslash L}|}\sum_{\substack{L'\subseteq L\\v\in L'}}(-1)^{|L\backslash L'|}\frac{1}{|\mathfrak{X}_{L\backslash L'}|}\sum_{\substack{j_{AB}\in\mathfrak{X}_{AB}\\j_{L'}=x_{L'}}}\log p_{A|B\backslash\{v\}}(j_{A}|j_{B\backslash\{v\}}) \nonumber \\ 
&+\frac{1}{|\mathfrak{X}_{AB\backslash L}|}\sum_{\substack{L'\subseteq L\\v\not\in L'}}(-1)^{|L\backslash L'|}\frac{1}{|\mathfrak{X}_{L\backslash L'}|}\sum_{\substack{j_{AB}\in\mathfrak{X}_{AB}\\j_{L'}=x_{L'}}}\log p_{A|B\backslash\{v\}}(j_{A}|j_{B\backslash\{v\}}) \nonumber \\ 
&=D_{1} + D_{2} \quad\textrm{(say)}.
\end{align}
Now
\begin{align}\label{eq5.5}
D_{2} &= \frac{1}{|\mathfrak{X}_{AB\backslash L}|}\sum_{\substack{L'\subseteq L\\v\not\in L'}}(-1)^{|L\backslash L'|}\frac{1}{|\mathfrak{X}_{L\backslash L'}|}\sum_{\substack{j_{AB}\in\mathfrak{X}_{AB}\\j_{L'}=x_{L'}}}\log p_{A|B\backslash\{v\}}(j_{A}|j_{B\backslash\{v\}}) \nonumber \\
&= \frac{1}{|\mathfrak{X}_{AB\backslash L}|}\sum_{\substack{L'\subseteq L\\v\not\in L'}}(-1)^{|L\backslash L'|}\frac{1}{|\mathfrak{X}_{L\backslash L'}|}\sum_{j_{v}\in\mathfrak{X}_{v}}\sum_{\substack{j_{AB}\in\mathfrak{X}_{AB}\\j_{L'\cup\{v\}}=x_{L'\cup\{v\}}}}\log p_{A|B\backslash\{v\}}(j_{A}|j_{B\backslash\{v\}}) \nonumber \\
&=\frac{1}{|\mathfrak{X}_{AB\backslash L}|}\sum_{\substack{L'\subseteq L\\v\not\in L'}}(-1)^{(|L\backslash L'\cup\{v\}|+|v|)}\frac{1}{|\mathfrak{X}_{L\backslash (L'\cup\{v\})}||\mathfrak{X}_{v}|}\times|\mathfrak{X}_{v}|\sum_{\substack{j_{AB}\in\mathfrak{X}_{AB}\\j_{L'\cup\{v\}}=x_{L'\cup\{v\}}}}\log p_{A|B\backslash\{v\}}(j_{A}|j_{B\backslash\{v\}}) \nonumber \\
&= -\frac{1}{|\mathfrak{X}_{AB\backslash L}|}\sum_{\substack{L'\subseteq L\\v\not\in L'}}(-1)^{|L\backslash L'\cup\{v\}|}\frac{1}{|\mathfrak{X}_{L\backslash (L'\cup\{v\})}|}\sum_{\substack{j_{AB}\in\mathfrak{X}_{AB}\\j_{L'\cup\{v\}}=x_{L'\cup\{v\}}}}\log p_{A|B\backslash\{v\}}(j_{A}|j_{B\backslash\{v\}}) \nonumber \\
&= -\frac{1}{|\mathfrak{X}_{AB\backslash L}|}\sum_{\substack{L'\subseteq L\\v\in L'}}(-1)^{|L\backslash L'|}\frac{1}{|\mathfrak{X}_{L\backslash L'}|}\sum_{\substack{j_{AB}\in\mathfrak{X}_{AB}\\j_{L'}=x_{L'}}}\log p_{A|B\backslash\{v\}}(j_{A}|j_{B\backslash\{v\}}) \nonumber \\
&=-D_{1}. 
\end{align}
From (\ref{eq5.4}) and (\ref{eq5.5}), we get $f(\lambda_{A|B})=0$. This completes the proof.
\end{proof}

\subsection*{Proof of Theorem \ref{th5.2}:} 
\begin{proof}
By Theorem \ref{th5.1}, we have
\begin{equation}\label{eq5.6}
\lambda_{L}^{M} = \lambda_{L}^{B} + f(\lambda_{A|B})
\end{equation}
for a smooth function $f$. Also, since $X_{A}\ci X_{v}|X_{B\backslash\{v\}}$ for some $v\in L$, we have $f(\lambda_{A|B})=0$. So
\begin{equation}\label{eq5.7}
\lambda_{L}^{M}(x_{L}) = \lambda_{L}^{B}(x_{L})\quad\forall~x_{L}\in\mathfrak{X}_{L}, 
\end{equation}
which implies that $\mathfrak{X}_{M}$ is collapsible into $\mathfrak{X}_{B}$ with respect to $\lambda_{L}^{M}$ for $L\subseteq M$. Now consider two complete MLL parameterizations of $\mathcal{F}$ on $\mathfrak{X}_{M}$ corresponding to the collections $\mathcal{S}$ and $\mathcal{T}$ (say). Let $\mathcal{S}$ be non-hierarchical, while $\mathcal{T}$ is hierarchical. Since $\mathcal{T}$ is both hierarchical and complete, $\tilde{\lambda}_{\mathcal{T}}$ is a smooth MLL parameterization of $\mathcal{F}$ on $\mathfrak{X}_{M}$ by Theorem 2 of Bergsma and Rudas (2002). The MLL parameters defined by (\ref{eq5.7}) are non-smooth by Theorem 3 of Bergsma and Rudas (2002), and are hence embedded not in $\mathcal{T}$ but in $\mathcal{S}$. Specifically, the effect $L$ is defined in $\mathcal{S}$ not within the first but some subsequent margin of which it is a subset. However, this is not the case with respect to $L$ in $\mathcal{T}$. Let $B$ be the first margin of which $L$ is a subset in $\mathcal{S}$ and $\mathcal{T}$. Then from (\ref{eq5.7}), $L$ is defined within $M$ instead of $B$ in $\mathcal{S}$, while it has to be defined within $B$ in $\mathcal{T}$. 

In general, if the conditional distribution $X_{A}|X_{B}$ is fixed, that is, $p_{A|B}$ or $f(\lambda_{A|B})$ is known, then the relationship between $\lambda_{L}^{B}$ and $\lambda_{L}^{M}$ is linear from (\ref{eq5.6}). Indeed, $\lambda_{L}^{B}$ and $\lambda_{L}^{M}$ become interchangeable as part of a parameterization, preserving smoothness. From (\ref{eq5.7}), $f$ is known since $f=0$. This implies $\lambda_{L}^{B}$ and $\lambda_{L}^{M}$ are interchangeable, that is, $\tilde{\lambda}_{\mathcal{S}}$ is smooth if and only if $\tilde{\lambda}_{\mathcal{T}}$ is also smooth, which is true. Thus $\tilde{\lambda}_{\mathcal{S}}$ provides a smooth parameterization of $\mathcal{F}$ on $\mathfrak{X}_{M}$ under collapsibility conditions thereby defining a curved exponential family. By assumption of conditional independence and Lemma \ref{lem1.3}, we have $\lambda_{A'vB'}^{M}=0$ for every $\emptyset\neq A'\subseteq A$ and $\emptyset\neq B'\subseteq B$. Hence  $\tilde{\lambda}_{\mathcal{S}}$ is given by $\{\tilde{\lambda}_{L}^{M}\mid L\in\mathbb{P}(M)\backslash\mathcal{D}\}$ where $\mathcal{D}$ is the collection of all sets of the form $A'vB'$ with $\emptyset\neq A'\subseteq A$ and $\emptyset\neq B'\subseteq B\backslash \{v\}$. 
\end{proof}

\subsection*{Proof of Theorem \ref{th5.3}:} 
\begin{proof}
In Theorem \ref{th5.1}, take $A=R$, $B=M\backslash R$ and $L\subseteq B$ (see (\ref{eq5.2})) . Also, note that $f(\lambda_{A|B})=0$ if $X_{R}\ci X_{v}|X_{(M\backslash R)\backslash\{v\}}$ for some $v\in M\backslash R$ so that from (\ref{eq5.2}), we have $\lambda_{L}^{M}(x_{L})=\lambda_{L}^{M\backslash R}(x_{L})$ for all $x_{L}\in\mathfrak{X}_{L}$. Hence, the result follows from Definition \ref{def2.1}.
\end{proof}


\begin{thebibliography}{99} 
	\bibitem{Agresti1990} Agresti, A., 1990. Categorical data analysis. Second edition. Wiley, New York.
	\bibitem{Andersen1990} Andersen, E. B., 1990. The statistical analysis of categorical data.  Berlin: Springer.
	\bibitem{Becker1994} Becker, M. P., 1994. Analysis of repeated categorical measurements using models for marginal distributions: an application to trends in attitudes on legalized abortion. Sociological Methodology. Blackwell, Oxford.
	\bibitem{Bergsma2002} Bergsma, W. P., Rudas, T., 2002. Marginal models for categorical data. Ann. Statist. 30(1), 140-159.
	\bibitem{Bergsma2009} Bergsma, W. P., Croon, M. A., Hagenaars, J. A., 2009. Marginal Models: For Dependent, Clustered, and Longitudinal Categorical Data. Springer Science \& Business Media, 2009.
	\bibitem{Bishop1975} Bishop, Y. M. M., Fienberg, S. E., Holland, P. W., 1975. Discrete multivariate analysis: theory and practice. Cambridge: MIT Press.
	\bibitem{Charalambides2002} Charalambides, A. C., 2002. Enumerative combinatorics. Florida: Chapman and Hall. 
	\bibitem{Cox2003} Cox, D. R., Wermuth, N., 2003. A general condition for avoiding effect reversal after marginalization. J. R. Stat. Soc. Ser. B 65, 937-941.
	\bibitem{Ducharne1986} Ducharne, G. R., Lepage, Y., 1986. Testing collapsibility in multidimensional tables. J. R. Stat. Soc. Ser. B 48, 197-205.
	\bibitem{Evans2013} Evans, R. J., Richardson, T. J., 2013. Marginal log-linear parameters for graphical Markov models. J. R. Stat. Soc. Ser. B 75, 743-768. 
	\bibitem{Evans2015} Evans, R. J., 2015. Smoothness of marginal log-linear parameterizations. Electron. J. Stat. 9, 475-491.
	\bibitem{Everitt1977} Everitt, B. S., 1977. Some properties of statistics used for measuring observer agreement in the recording of signs. British J. Math. Statist. Psych. 30, 227-233.
	\bibitem{Forcina2010} Forcina, A., Lupparelli, M., Marchetti, G. M., 2010. Marginal parameterizations of discrete models defined
	by a set of conditional independencies. J. Multivariate Anal. 101, 2519-2527.
	\bibitem{Glonek1995} Glonek, G. F. V., McCullagh, P., 1995. Multivariate logistic models. J. R. Stat. Soc. Ser. B 57, 533-546.
	\bibitem{Glonek1996} Glonek, G. F. V., 1996. A class of regression models for multivariate categorical responses. Biometrika 83, 15-28.
	\bibitem{Guo1995} Guo, J. H., Geng, Z., 1995. Collapsibility of logistic regression coefficients. J. R. Stat. Soc. Ser. B 57, 263-267.
	\bibitem{Kauermann1997} Kauermann, G., 1997. A note on multivariate logistic models for contingency tables. Austral. J. Statist.
	39, 261-276.
	\bibitem{Lang1994} Lang, J. B., Agresti, A., 1994. Simultaneously modelling joint and marginal distributions of multivariate categorical responses. J. Amer. Statist. Assoc. 89, 625-632.
	\bibitem{Liang1992} Liang, K. Y., Zeger, S. L., Qaqish, B., 1992. Multivariate regression analyses for categorical data (with discussion). J. R. Stat. Soc. Ser. B 54, 3-40.
	\bibitem{McCullagh1989} McCullagh, P., Nelder, J. A., 1989. Generalized Linear Models. 2nd ed. Chapman and Hall, London.
	\bibitem{Rudas2010} Rudas, T., Bergsma, W. P., Nem\'{e}th, R., 2010. Marginal log-linear parameterization of conditional independence models. Biometrika 94, 1006-1012.
	\bibitem{Simpson1951} Simpson, E. H., 1951. The interpretation of interaction in contingency tables. J. R. Stat. Soc. Ser. B 13, 238-241.
	\bibitem{Vellaisamy2012} Vellaisamy, P., 2012. Simpson's paradox and collapsibility. J. Indian Statist. Assoc., 50, 297-317.
	\bibitem{Vellaisamy2007} Vellaisamy, P., Vijay, V., 2007. Some collapsibility results for $n$-dimensional contingency tables. Ann. Inst. Statist. Math. 59, 577-576.
	\bibitem{Vellaisamy2009} Vellaisamy, P., Vijay, V., 2009. Log-linear modelling using conditional log-linear structures. Ann. Inst. Statist. Math. 61, 309-329.
	\bibitem{Vellaisamy2010} Vellaisamy, P., Vijay, V., 2010. Collapsibility of contingency tables based on conditional models. J. Statist. Plann. Inference 140, 1243-1255.
	\bibitem{Wermuth1987} Wermuth, N., 1987. Parametric collapsibility and the lack of moderating effects in contingency tables with a dichotomous response variable. J. R. Stat. Soc. Ser. B 49, 353-364.
	\bibitem{Whittemore1978} Whittemore, A. S., 1978. Collapsibility of multidimensional contingency tables. J. R. Stat. Soc. Ser. B 40, 328-340. 
\end{thebibliography}
\end{document}